\documentclass[12pt]{amsart}
\usepackage{mathrsfs}
\usepackage{amssymb}

\usepackage{amscd}
\usepackage{amsmath, amssymb}
\usepackage{amsfonts}
\usepackage[colorlinks,linkcolor=blue,citecolor=blue, pdfstartview=FitH]{hyperref}

\numberwithin{equation}{section}

\theoremstyle{plain}
\newtheorem{thm}{Theorem}[section]
\newtheorem{theorem}[thm]{Theorem}
\newtheorem{lemma}[thm]{Lemma}
\newtheorem{corollary}[thm]{Corollary}
\newtheorem{proposition}[thm]{Proposition}

\theoremstyle{definition}

\newtheorem{remark}[thm]{Remark}

\newtheorem{definition}[thm]{Definition}

\newtheorem{example}[thm]{Example}

\newtheorem{defn-thm}[thm]{Definition-Theorem}

\newcommand{\sL}{{\mathcal L}}

\newcommand{\Ker}{{ Ker}}

\newcommand{\bp}{\bar{\partial}}

\newcommand{\btheorem}{\begin{theorem}}
\newcommand{\etheorem}{\end{theorem}}
\newcommand{\bproposition}{\begin{proposition}}
\newcommand{\eproposition}{\end{proposition}}
\newcommand{\bdefinition}{\begin{definition}}
\newcommand{\edefinition}{\end{definition}}
\newcommand{\bcorollary}{\begin{corollary}}
\newcommand{\ecorollary}{\end{corollary}}
\newcommand{\bproof}{\begin{proof}}
\newcommand{\eproof}{\end{proof}}
\newcommand{\bremark}{\begin{remark}}
\newcommand{\eremark}{\end{remark}}
\newcommand{\eexample}{\end{example}}
\newcommand{\bexample}{\begin{example}}

\newcommand{\elemma}{\end{lemma}}
\newcommand{\blemma}{\begin{lemma}}

\newcommand{\p}{\partial}

\renewcommand{\bar}{\overline}

\renewcommand{\phi}{\varphi}

\newcommand{\ee}{\end{eqnarray*}}
\newcommand{\be}{\begin{eqnarray*}}

\newcommand{\beq}{\begin{equation}}
\newcommand{\eeq}{\end{equation}}

\newcommand{\bd}{\begin{enumerate}}
\newcommand{\ed}{\end{enumerate}}

\renewcommand{\tilde}{\widetilde}

\usepackage{fancyhdr}
\begin{document}
\title{Solving equations with Hodge theory}
\makeatletter
\let\uppercasenonmath\@gobble
\let\MakeUppercase\relax
\let\scshape\relax
\makeatother

\author{Kefeng Liu}
\address{Kefeng Liu, School of Mathematics,
Capital Normal University, Beijing, 100048, China}
\address{Department of Mathematics, University of California at Los Angeles, California 90095}

\email{liu@math.ucla.edu}

\author{Shengmao Zhu}
\address{Shengmao Zhu, Center of Mathematical Sciences,
Zhejiang University, Hangzhou, 310027, China}

\email{szhu@zju.edu.cn, shengmaozhu@126.com} \maketitle
\begin{abstract} We treat two quite different problems related to changes of complex structures on K\"ahler manifolds by using global geometric method.
 First, by using operators from Hodge
theory on compact K\"ahler manifold, we present a closed explicit extension
formula for holomorphic  canonical forms in different complex
structures. As applications, we give a closed explicit formula for certain
canonical sections of Hodge bundles on marked and polarized moduli
spaces of projective manifolds, and provide a closed explicit extension formula for holomorphic pluricanonical
forms under certain natural conditions. Second, by using the
operators in $L^2$-Hodge theory on Poincar\'e disk, we present a
simple and unified method to solve the Beltrami equations with measurable coefficients for
quasi-conformal maps.
\end{abstract}

\maketitle 
%

\section{Introduction}
Let $M$ be a complex manifold of complex dimension
$\dim_{\mathbb{C}}M=n$, and $\varphi$ is a Beltrami differential
which is a tangent bundle valued $(0,1)$-form in $A^{0,1}(M,
T^{1,0}M)$. See Section \ref{subsection-Beltrami} for a detailed
discussion of complex structures and Beltrami differentials.  If the
Beltrami differential $\varphi$ is integrable in the sense that
$$\bp\phi =\frac{1}{2}[\phi, \phi],$$
then $\varphi$ determines a new complex structure on $M$, denoted by
$M_\phi$.

Given any differential form $\sigma$ on $M$, let us denote by
$\phi\lrcorner \sigma= i_\phi\sigma= \phi \sigma$ the natural
contraction morphism throughout this paper, and we define a map
\begin{align} \label{exponentialphi}
\rho_{\varphi}(\sigma)= e^{i_\phi}\sigma= \sum_{k\geq
0}\frac{1}{k!}i_{\varphi}^k\sigma.
\end{align}
Then $\rho_\varphi$ gives a bijection from $A^{n,0}(M)$ to
$A^{n,0}(M_\varphi)$. Therefore, given a smooth $(n,0)$-form
$\Omega$ on complex manifold $M$, $\rho_\varphi(\Omega)$ is a
$(n,0)$-form on $M_\phi$. We will show that
$\rho_\varphi(\Omega)=e^{i_\phi}\Omega$  is holomorphic  on $M_\phi$
if and only if
\begin{align} \label{obstructionequation-intro}
\bp\Omega+\p(\phi\lrcorner \Omega)=0.
\end{align}

In this paper we present a global geometric method to study the changes of complex structures on K\"ahler manifolds and discuss the related geometric and analytic applications without using the local deformation theory as developed by Kodaira-Spencer. The following is a brief outline of the method and main results of
the paper.
\subsection{Extension formula for holomorphic canonical forms on compact K\"ahler manifold}
We consider compact K\"ahler manifold $(M,\omega)$ of complex
dimension $n$ with K\"ahler form $\omega$. Let
$\varphi\in A^{0,1}(M,T^{1,0}M)$ be an integrable Beltrami
differential. Given a holomorphic $(n,0)$-form $\Omega_0$ on $M$,
i.e. a holomorphic section of the canonical bundle $K_M$, our first goal
is to construct a holomorphic $(n,0)$-form $\Omega(\varphi)$ on
$M_\varphi$, such that $\Omega(0)=\Omega_0$. We construct such
$\Omega(\varphi)$ from the solution of extension equation
(\ref{obstructionequation-intro}), which will be solved by using
Hodge theory on $(M,\omega)$.

More precisely, we introduce the operator
$$ T=\bp^*G\p$$
where $d=\p+\bp$, with $\p$ and $\bp$ the $(1,0)$ and $(0,1)$
differentials, $\p^*$ and $\bp^*$ the corresponding adjoint
operators, and $G$ denotes the Green operator of the Laplacian
operator $\Box_{\bp}$. Then we will show that, for any $g\in
A^{p,q}(M)$,
\begin{align} \label{quasiisometry-intro}
\|\bar{\partial}^* G\partial g\|^2\leq \|g\|^2,
\end{align}
where $\|\cdot \|$ denotes the $L^2$-norm induced by the K\"ahler
metric $\omega$. (\ref{quasiisometry-intro}) implies that $T$ is an
operator of norm $||T||\leq1$ on the Hilbert space of $L^2$-forms.
Let $\phi$ be a Beltrami differential with $L_\infty$-norm
$||\phi||_\infty<1$, then as a corollary we see that the operator
$I+T\phi$ is invertible, where $\phi$ is considered as a contraction
operator on the Hilbert space of $L^2$-forms.

We will show in Section \ref{Section-extensionforms} that
$$
\Omega=(I+T\varphi)^{-1}\Omega_0
$$
is a solution to the equation (\ref{obstructionequation-intro}).
Then, we obtain
\begin{theorem}\label{extensionprop-intro} Given an integrable Beltrami differential $\phi$
such that the $L_\infty$-norm $||\phi||_{\infty}<1$, and any
holomorphic $(n,0)$-form $\Omega_0$ on $M$, then
\begin{align} \label{extension-intro}
\Omega(\varphi)=\rho_\varphi((I+T\phi)^{-1}\Omega_0)
\end{align}
 is a holomorphic $(n,0)$-form on $M_\phi$. In particular,
 $\Omega(0)=\Omega_0$.
\end{theorem}

Therefore, (\ref{extension-intro}) is the explicit closed extension formula which we
are looking for. Note that the above construction is global in the
sense that it does not depend on the local deformation theory of
Kodaira-Spencer and Kuranishi. On the other hand, Theorem
\ref{extensionprop-intro} can be applied to the integrable Beltrami
differential $\phi(t)$ constructed from the
Kodaira-Spencer-Kuranishi deformation theory with $|t|\le
\varepsilon$ small, such that $\|\varphi(t)\|_{\infty}<1$.
\begin{corollary}\label{extensioncoro-intro}
Let $\pi: \mathcal{X}\rightarrow \Delta_\varepsilon
\subset{\mathbb{C}^m}$ be the Kuranishi family of compact K\"ahler
manifolds with $M_t=\pi^{-1}(t)=M_{\varphi(t)}$, where $t\in
\Delta_\varepsilon$. Given any holomorphic $(n,0)$-form $\Omega_0\in
A^{n,0}(M_0)$, we have that
\begin{align*}
\Omega(t)=\rho_t((I+T\varphi(t))^{-1}\Omega_0)
\end{align*}
 is a holomorphic $(n,0)$-form on $M_t$, where we denote by
 $\rho_t=\rho_{\varphi(t)}$. In particular, $\Omega(0)=\Omega_0$.
\end{corollary}
\begin{remark}
By the construction of the integrable Beltrami differential
$\varphi(t)$ in Kodaira-Spencer-Kuranishi deformaiton theory
\cite{MK71}, $\varphi(t)=\sum_{\mu\geq 1}\varphi_\mu(t)$, with
$\varphi_1(t)=\sum_{i=1}^m\eta_it_i$, where $\{\eta_i\}$ is a basis
for the harmonic space $\mathbb{H}^{1}(M,T^{1,0}M)$. Now the holomorphic
$(n,0)$ form on $M_t$ is given by $\Omega(t)$, whose first two terms
are
\begin{align} \label{twoterm-intro}
\Omega(t)=\Omega_0+\sum_{i=1}^m (\eta_i\lrcorner
\Omega_0-T(\eta_i\lrcorner \Omega_0))t_i+O(t^2).
\end{align}
From (\ref{twoterm-intro}), one can easily derive the curvature
formula of the $L^2$ metric on the corresponding Hodge bundle.
\end{remark}

\subsection{Closed formula for canonical section of Hodge bundle}
We then consider certain global canonical sections of holomorphic
forms on the moduli spaces of marked and polarized projective
manifolds of complex dimension $n$. Fixing a base point $M$ in the
moduli space $\mathcal{M}$ of marked and polarized projective
manifolds of complex dimension $n$, let $\omega_0$ be a K\"ahler
form on $M$. Here we only consider the connected component of the
 moduli space containing $M$.  Note that the de Rham cohomology group  gives
 us a trivial bundle $H^n(M)$ on $\mathcal{M}$ induced by the markings.

We will show that if the complex structure of any other point $M_1$
in $\mathcal{M}$ can be tamed by the same K\"ahler  form $\omega_0$
considered as a symplectic form, then there is a natural
construction of Beltrami differential $\phi$  on $M$ with its
$L_\infty$-norm $||\phi||_\infty<1$, such that $M_1=M_{\phi}$. By
using a theorem of Moser \cite{Moser} and the above operator $T$
from Hodge theory on $M$, we can deduce the following result which
gives a closed formula of a canonical section of the Hodge bundle of
holomorphic $n$-forms on $\mathcal{M}$.

\btheorem Given a point $(M,\omega_0)$ in the marked and polarized
moduli space $\mathcal{M}$ and a holomorphic $n$-form $s_0$ on $M$,
there is a canonical section $s$ of the Hodge bundle
$\mathcal{H}^{n,0}$, such that for any point $M_1$ in $\mathcal{M}$,
the de Rham cohomology class of $s$ in $H^n(M_1)$ is represented by
$$s = \rho_{\varphi}((I+T\phi)^{-1}s_0).$$
where $\phi$ is the Beltrami differential associated to $M_1$, and
$T=\bp^*G\p$ is the operator of the Hodge theory on $M$ with
K\"ahler metric $\omega_0$. \etheorem

\subsection{Extension formula for pluricanonical form}
Next we generalize the previous method to construct the extensions
of pluricanonical forms. Let $(M,\omega)$ be a compact K\"ahler
manifold of complex dimension $\dim_{\mathbb{C}}M=n$ with K\"ahler
form $\omega$, and $ \varphi\in A^{0,1}(M,T^{1,0}M)$ be an
integrable Beltrami differential. Let $m\geq 2$ be a fixed integer.
We consider a pluricanonical form $\sigma_0$ which is a holomorphic
section of $K_{M}^{\otimes m}$ over $M$, where $K_M$ denotes the
canonical line bundle of $M$. An important question is how to construct
the pluricanonical forms $\sigma(\varphi)$ on $M_\varphi$, such that
$\sigma(0)=\sigma_0$. In fact, for projective manifolds, the existence of extension was proved by Y.-T. Siu. In general  there is a famous conjecture due to
 Siu \cite{Siu}, about the invariance of plurigenera for compact
K\"ahler manifolds.

In our approach, Siu's conjecture is reduced to solving the
extension equation (\ref{extensionequ}). By using Hodge theory, we
provide a closed explicit formula for the solution of this extension equation
under certain conditions.

More precisely, let $(\mathcal{L},h)$ be an Hermitian holomorphic
line bundle over $(M,\omega)$.  Let $\nabla=\nabla'+\bp$ be the
Chern connection of $(\mathcal{L},h)$ with curvature $\Theta$.
We introduce the operator
\begin{align*}
T^{\nabla'}=\bp^* G \nabla'
\end{align*}
where $G$ is the Green operator associated to the Laplacian
$\overline{\square}=\bp\bp^*+\bp^*\bp$.

In particular,  we consider the holomorphic line bundle
$\mathcal{L}_M=K_{M}^{\otimes (m-1)}$ over  $(M,\omega)$ with the
induced Hermitian metric $h_\omega=\det(g)^{-(m-1)}$, where $g$
denotes the K\"ahler metric matrix associated to the K\"ahler form
$\omega$.

We establish the following result in Section
\ref{Section-pluricanonical}.
\begin{theorem} \label{introduction-theoremglobal}
Suppose $(\mathcal{L}_M,h_\omega)$ is a positive line bundle over a
compact K\"ahler manifold $(M,\omega)$  with curvature
$\sqrt{-1}\Theta=\rho\omega$ for a constant $\rho>0$, let
$\varphi\in A^{0,1}(M,T^{1,0}M)$ be an integrable Beltrami
differential satisfying two conditions $div\varphi=0$ and
$L_\infty$-norm $\|\varphi\|_\infty<1$. Then, for any holomorphic
pluricanonical form $\sigma_0\in A^{n,0}(M,\mathcal{L}_M)$,
\begin{align*}
\sigma(\varphi)=\rho_\varphi((I+T^{\nabla'}\varphi)^{-1}\sigma_0)
\end{align*}
is a holomorphic pluricanonical form in
$A^{n,0}(M_\varphi,\mathcal{L}_{M_\varphi})$.
\end{theorem}

Note that Theorem \ref{introduction-theoremglobal}  is global in the
sense that it does not depend on the local deformation family.
Theorem \ref{introduction-theoremglobal} can be used to construct
the closed extension formula for pluricanonical forms of
K\"ahler-Einstein manifold of general type, see Definition
\ref{def-KEgeneraltype}.

Let $\pi: \mathcal{X}\rightarrow B_\varepsilon \subset{\mathbb{C}}$
be a holomorphic family of compact K\"ahler-Einstein manifolds of
general type. For $t\in B_\varepsilon$, we assume
$M_t=\pi^{-1}(t)=M_{\varphi(t)}$, where $\varphi(t)\in
A^{0,1}(M_0,T^{1,0}M_0)$ denotes an integrable Beltrami differential
satisfying the Kuranishi gauge $\bp^*\varphi(t)=0$.

 As an application of Theorem
\ref{introduction-theoremglobal}, we obtain
\begin{corollary} \label{invarianceofKE}
Given any holomorphic pluricanonical form $\sigma_0\in
A^{n,0}(M_0,\mathcal{L}_{M_0})$, then
\begin{align} \label{closedformKE}
\sigma(t)=\rho_t((I+T^{\nabla'}\varphi(t))^{-1}\sigma_0)
\end{align}
is a holomorphic pluricanonical form in
$A^{n,0}(M_t,\mathcal{L}_{M_t})$ with $\sigma(0)=\sigma_0$, where
$\rho_t=\rho_{\varphi(t)}$.
\end{corollary}
Corollary \ref{invarianceofKE} implies the invariance of plurigenera
for K\"ahler-Einstein manifolds of general type, which has been
obtained in \cite{Sun}. Formula (\ref{closedformKE}) provides a simple
closed explicit formula for the extension of pluricanonical form.

\subsection{Solving the Beltrami equation} Beltrami equation is very important in the
development of complex analysis and moduli theory of Riemann surfaces. It also has many
important applications in other subjects. See, for examples
\cite{Ahlfors}, \cite{Bojarski} and \cite{Gut}.

Given a  measurable function  $\mu_0$ on the unit disc $D\subset
\mathbb{C}$, suppose $\mathrm{sup}\,|\mu_0|<1$, let $\mu=\mu_0
\frac{\p}{\p z} \otimes d\bar{z}$ be a Beltrami differential on $D$
with coordinate $z$.  Recall that solving the Beltrami equation is
to find a function $f$ on the unit disc $D$, such that $$\bp f  =
\mu\p f.$$ Our observation is that the Beltrami equation can be
solved by using the $L^2$-Hodge theory. We will show in Section
\ref{Section-L2} that the $L^2$-Hodge theory holds on disk $D$ with
the Poinc\'are metric $\omega_P$. So we also have the operator
$T=\bp^*G\p$ with norm $\|T\|\leq 1$.

Note that the $L_\infty$-norm of $\mu$ is independent of Hermitian
metric on $D$ and is equal to $\mathrm{sup}\, |\mu_0|$, i.e.
$\|\mu\|_\infty<1$.  Similarly, we show that for a holomorphic
 one form $h_0$ on $D$, the equation  $$\bp h =- \p \mu h$$ has a solution
$$h= (I + T\mu)^{-1}  h_0.$$
As a corollary we can directly get the a solution
of the Beltrami equation for any measurable $\mu_0$. In particular we have,

\btheorem  Assume that $||\mu||_\infty=\mathrm{sup}\, |\mu_0|<1$, if
$\mu_0$ is of regularity $C^k$,  then the Beltrami equation
$$\bp f  = \mu\p f$$ has a solution $w(z)$ of regularity $C^{k+1}$.
\etheorem

The rest of this paper is organized as follows.  In Section
\ref{Section-extensionequation}, we will first review the basics of
operators on differential forms, Beltrami differentials and
extension equations which are needed for our discussions. Then in
Section \ref{Section-Hodgetheory}, we briefly review the Hodge
theory on compact K\"ahler manifold, introduce the operator $T$ and
discuss the quasi-isometry formula.  In Section
\ref{Section-extensionforms} we write down and prove a closed
formula for extension of holomorphic canonical form on compact
K\"ahler manifold in new complex structures. In Section
\ref{Section-moduli} we present a closed formula for a global
canonical section of the Hodge bundle of holomorphic form on moduli space.

In Section \ref{Section-pluricanonical}, we generalize the methods
in Sections \ref{Section-Hodgetheory}, \ref{Section-extensionforms}
to construct the closed explicit formula for extension of holomorphic
pluricanonical forms under certain conditions.

In Section \ref{Section-L2} we review and prove a few basic facts of
$L^2$-Hodge decomposition theory. In particular we will present a
general result from \cite{Gromov} that $L^2$-Hodge theory holds on
the universal cover of a K\"ahler hyperbolic manifold which is
possibly known and of independent interest for other applications,
although in this paper we only need the case when the universal
cover is the unit ball with standard Poincar\'e hyperbolic metric.
In this section we will also discuss briefly the relationship
between the $L^2$-Hodge theory and the $L^2$-estimate of
H\"ormander. We show that $L^2$-Hodge theory is more general and
implies the $L^2$-estimate. In Section \ref{Beltrami} we apply the
results in  Section \ref{Section-L2} to solve the Beltrami
equations. In Section \ref{general} we discuss various applications
and extensions of our method.

\textbf{Acknowledgements.} This paper grew out of several lectures the
first author presented in the complex geometry seminar we organized
in School of Mathematical Science, Capital Normal University during the academic year
 2016-2017. The first author would like to thank all the
participants of the seminar for their interest. The research of the
first author is supported by NSFC (Grant No. 11531012) and NSF. The
second author would like to thank CSC to support his visiting in
UCLA.

\section{Extension equations} \label{Section-extensionequation}
In this section, we first review some basic results about the
operators on differential forms following \cite{LRY}. Then we
introduce the definitions of Beltrami differentials and
integrability condition. We derive the extension equation for
constructing the holomorphic canonical form on the new complex manifold
$M_\varphi$, which is determined by an integrable Beltrami
differential $\varphi$.
\subsection{Generalized Cartan formulas}
Let $M$ be a complex manifold of dimension $n$. Let $\varphi\in
A^{0,k}(M,T^{1,0}M)$ be a $T^{1,0}M$-value $(0,k)$-form. We
introduce the contraction operator
\begin{align*}
i_{\varphi}: A^{p,q}(M)\rightarrow A^{p-1,q+k}(M)
\end{align*}
as in \cite{LRY}. If we write $\varphi=\eta\otimes Y$ with $\eta\in
A^{0,k}(M)$ and $Y\in C^{\infty}(T^{1,0}M)$, then for $\sigma\in
A^{p,q}(M)$,
\begin{align*}
i_{\varphi}(\sigma)=\eta\wedge i_Y\sigma.
\end{align*}
Sometimes, we also use the notations $\varphi\lrcorner
\eta=\varphi\eta$ to denote the contraction $i_{\varphi} \eta$
alternatively. By definition, we have
\begin{align*}
i_{\varphi}i_{\varphi'}=(-1)^{(k+1)(k'+1)}i_{\varphi'}i_{\varphi}
\end{align*}
if $\varphi\in A^{0,k}(M)$ and $\varphi' \in A^{0,k'}(M)$. The Lie
derivation of $\mathcal{\varphi}$ is defined by
\begin{align*}
\mathcal{L}_{\varphi}=(-1)^kd\circ i_\varphi+i_\varphi \circ d
\end{align*}
which can be decomposed into the sum of two parts
\begin{align*}
\mathcal{L}_{\varphi}^{1,0}=(-1)^k\partial \circ i_\varphi
+i_\varphi \circ \partial, \ \
\mathcal{L}_{\varphi}^{0,1}=(-1)^k\bar{\partial}\circ i_\varphi
+i_\varphi \circ\bar{\partial}.
\end{align*}
The Lie bracket of $\varphi$ and $\varphi'$ is defined by
\begin{align*}
[\varphi,\varphi']=\sum_{i,j=1}^n\left(\varphi^i \wedge
\partial_i \varphi'^j-(-1)^{kk'}\varphi'^i\wedge \partial_i
\varphi^j\right)\otimes \partial_j,
\end{align*}
if  $\varphi=\sum_{i}\varphi^i\partial_{i}\in A^{0,k}(M,T^{1,0}M)$
and $\varphi'=\sum_{i}\varphi'^i\partial_{i}\in
A^{0,k'}(M,T^{1,0}M)$.

We have the following generalized Cartan formula \cite{LR,LRY} which
can be proved by direct computations.
\begin{lemma}
For any $\varphi, \varphi'\in A^{0,1}(M,T_M^{1,0})$, then on $
A^{*,*}(M)$,
\begin{align} \label{genralizedcartan}
i_{[\varphi,\varphi']}=\mathcal{L}_{\varphi}\circ
i_{\varphi'}-i_{\varphi'}\circ\mathcal{L}_{\varphi},
\end{align}
\end{lemma}

Let $\sigma\in A^{*,*}(M)$. By applying the formula
(\ref{genralizedcartan}) to $\sigma$ and considering the types,  we
immediately obtain
\begin{align} \label{genralizedcartan_special}
[\varphi,\varphi]\lrcorner \sigma=2\varphi\lrcorner
\partial\varphi\lrcorner \sigma -\partial(\varphi\lrcorner
\varphi\lrcorner\sigma)-\varphi\lrcorner\varphi\lrcorner
\partial\sigma.
\end{align}

\subsection{Beltrami differentials} \label{subsection-Beltrami}
In this section, $M$ is a complex manifold with
$\dim_{\mathbb{C}}M=n$, and we denote by $X$ the underlying real
manifold of $M$ of real dimension $2n$. The associated almost
complex structure of the complex manifold $M$ gives a direct sum
decomposition of the complexified tangent bundle,
\begin{align*}
T_{\mathbb{C}}X=T^{1,0}M\oplus T^{0,1}M.
\end{align*}
Let $J$ be another almost complex structure on $X$. Then, $J$ gives
another direct sum decomposition,
\begin{align*}
T_{\mathbb{C}}X=T^{1,0}M_{J}\oplus T^{0,1}M_J.
\end{align*}
Denote by
\begin{align*}
\iota_1: T_{\mathbb{C}}X\rightarrow T^{1,0}M, \  \iota_2:
T_{\mathbb{C}}X\rightarrow T^{0,1}M,
\end{align*}
the two projection maps.

\begin{definition} [cf. Definition 4.2 \cite{Kir} ]
Let $J$ be an almost complex structure on $X$, we say that $J$ is of
finite distance from the given complex structure $M$ on $X$, if the
restriction map
\begin{align*}
\iota_1|_{T^{1,0}M_J}: T^{1,0}M_J\rightarrow T^{1,0}M
\end{align*}
is an isomorphism.
\end{definition}
Therefore, if $J$ is of finite distance from $M$, one can define a
map
\begin{align*}
\bar{\varphi}: T^{1,0}M\rightarrow T^{0,1}M
\end{align*}
by setting $$\bar{\varphi}(v)=-\iota_2 \circ
\left(\iota_1|_{T^{1,0}M_J}\right)^{-1}(v).$$This map is
well-defined since $\iota_1|_{T^{1,0}M_J}$ is an isomorphism. It is
clear that
\begin{align*}
T^{1,0}M_J=\{v-\bar{\varphi}(v)|v\in T^{1,0}M \},\
T^{0,1}M_J=\{v-\varphi(v)|v\in T^{0,1}M \},
\end{align*}
and their corresponding dual spaces are
\begin{align} \label{dualspace}
\Lambda^{1,0}M_J=\{w+\varphi(w)|w\in \Lambda^{1,0}M \},\
\Lambda^{0,1}M_J=\{w+\bar{\varphi}(w)|w\in \Lambda^{0,1}M \}.
\end{align}
In this way, $\varphi$ determines a $T^{1,0}M$-valued $(0,1)$-form
which is also denoted by $\varphi\in A^{0,1}(M,T^{1,0}M)$ for
convenience. By the condition
\begin{align*}
T^{1,0}M\oplus T^{0,1}M=T_{\mathbb{C}}X=T^{1,0}M_J\oplus T^{0,1}M_J,
\end{align*}
the transformation matrix
\begin{equation*}
\left(
\begin{array}{ll}
I_n & -\bar{\varphi}  \\
-\varphi & I_n
\end{array}\right)
\end{equation*}
from a basis of $T^{1,0}M\oplus T^{0,1}M$ to a basis of
$T^{1,0}M_J\oplus T^{0,1}M_J$ must be nondegenerate. Therefore $
\det(I_n-\varphi\bar{\varphi})\neq 0. $  In fact, we have
\begin{proposition} [cf. Proposition 4.3  \cite{Kir}] \label{propfinitedistance}
There is a bijective correspondence between the set of almost
complex structures of finite distance from $M$ and the set of all
$\varphi\in A^{0,1}(M,T^{1,0}M)$ such that, at each point $p\in X$,
the map $\varphi\bar{\varphi}$ does not have eigenvalue 1.
\end{proposition}

\begin{definition}
If $\varphi\in A^{0,1}(M,T^{1,0}M)$ satisfies the condition in
Proposition \ref{propfinitedistance}, we say that $\varphi$ is a
Beltrami differential. If $\varphi$ satisfies the integrability
condition
$$\bar{\partial}\varphi=\frac{1}{2}[\varphi,\varphi],$$ we call
$\varphi$ an integrable Beltrami differential.
\end{definition}
So a Beltrami differential $\varphi$ determines an almost complex
structure of finite distance from $M$. We denote the corresponding
almost complex structure (i.e. almost complex manifold) by
$M_\varphi$. An integrable Beltrami differential $\varphi$ gives a
new complex structure  on $X$ by the Newlander-Nirenberg theorem
\cite{NN}, the corresponding complex manifold is denoted by
$M_\varphi$.

\subsection{Extension equations} \label{subsection-extension}
Given a Beltrami differential $\varphi$, for any $x\in X$, we can
pick a local holomorphic coordinate $(U,z^1,...,z^n)$ near $x$. Then
by $(\ref{dualspace})$,
\begin{align*}
\Lambda^{1,0}_x(M_{\varphi})=\text{Span}_{\mathbb{C}}\{dz^1+\varphi
dz^1,...,dz^n+\varphi dz^n\},
\end{align*}
and for any $1\leq p\leq n$,
\begin{align*}
\Lambda^{p,0}_x(M_{\varphi})=\text{Span}_{\mathbb{C}}\{(dz^{i_1}+\varphi
dz^{i_1})\wedge\cdots \wedge (dz^{i_p}+\varphi dz^{i_p})|1\leq
i_1<\cdots< i_p\leq n\}.
\end{align*}

Considering  the operator $e^{i_\varphi}$ defined by formula
(\ref{exponentialphi}), through a straightforward computation we
obtain
\begin{align*}
e^{i_\varphi}(dz^{i_1}\wedge \cdots \wedge dz^{i_p})=
(dz^{i_1}+\varphi dz^{i_1})\wedge \cdots\wedge (dz^{i_p}+\varphi
dz^{i_p})\in \Lambda_x^{p,0}(M_\varphi).
\end{align*}
Therefore, $e^{i_\varphi}$ is a map from $A^{p,0}(M)$ to
 $A^{p,0}(M_\varphi)$.
Moreover, it is easy to check that $$e^{i_\varphi}:\,
A^{p,0}(M)\rightarrow  A^{p,0}(M_\varphi) $$ is a bijection.

By the generalized Cartan formula (\ref{genralizedcartan}), it
follows that
\begin{proposition} \cite{Clemens,LRY}
Let $\varphi \in A^{0,1}(M,T^{1,0}M)$, then on $A^{*,*}(M)$, we have
\begin{align} \label{commuted}
e^{-i_\varphi}\circ d\circ
e^{i_\varphi}=d-\mathcal{L}_{\varphi}-i_{\frac{1}{2}[\varphi,\varphi]}.
\end{align}
\end{proposition}
\begin{proof}
Let $\eta\in A^{*,*}(M)$, formula (\ref{genralizedcartan}) implies
\begin{align} \label{genralizedcartan1}
d i_\varphi^2\eta=2i_\varphi  d i_\varphi\eta-i_\varphi^2
d\eta-i_{[\varphi,\varphi]}\eta.
\end{align}
Substituting $\eta$ with $i_\varphi \eta$ in
(\ref{genralizedcartan1}), we obtain
\begin{align*}
d(i_\varphi^3 \eta)&=2i_\varphi d(i_\phi^2 \eta)-i_\varphi^2
di_\varphi \eta-i_{[\varphi,\varphi]}i_\varphi\eta\\\nonumber
&=3i_\varphi^2 d(i_\varphi \eta)-2i_\varphi^3d\eta-3i_\varphi
i_{[\varphi,\varphi]}\eta.
\end{align*}
where we have used $i_\varphi
i_{[\varphi,\varphi]}=i_{[\varphi,\varphi]}i_{\varphi}$. Then by
induction, we immediately have
\begin{align*}
d(i_\varphi^k \eta)=ki_\varphi^{k-1}d(i_\varphi
\eta)-(k-1)i_\varphi^kd\eta-\frac{k(k-1)}{2}i_\varphi^{k-2}i_{[\varphi,\varphi]}\eta
\end{align*}
for $k\geq 2$. Through a straightforward computation, it is easy to
show that
\begin{align*}
d(e^{i_\varphi}\eta)=e^{i_\varphi}\left(d\eta-\mathcal{L}_{\varphi}\eta-i_{\frac{1}{2}[\varphi,\varphi]}\eta\right).
\end{align*}
The proof of (\ref{commuted}) is completed.
\end{proof}

\begin{corollary} \label{commuted1}
If $\varphi\in A^{0,1}(M,T^{1,0}M)$ is integrable, then on
$A^{*,*}(M)$, we have
\begin{align} \label{obstructionequation1}
e^{-i_\varphi}\circ d\circ
e^{i_\varphi}=d-\mathcal{L}_{\varphi}^{1,0}=d+\partial
i_\varphi-i_{\varphi} \partial.
\end{align}
\end{corollary}
\begin{proof}
By a straightforward computation, we have $
\bar{\partial}i_{\varphi}-i_{\varphi}\bar{\partial}=i_{\bar{\partial}\varphi}.
$ Then Corollary \ref{commuted1} follows directly from the
integrability of $\varphi$, i.e. $\bp
{\varphi}=\frac{1}{2}[\varphi,\varphi]$.
\end{proof}
In particular, we have
\begin{corollary} \label{coro-obstructionequation2}
Given an integrable Beltrami differential $\varphi\in
A^{0,1}(M,T^{1,0}M)$, for any $(n,0)$-form $\Omega$ on $M$, the
corresponding $(n,0)$-form
$\rho_{\varphi}(\Omega)=e^{i_{\varphi}}(\Omega)$ on $M_\varphi$ is
holomorphic, if and only if
\begin{align} \label{obstructionequation2}
\bar{\partial}\Omega =-\partial(\varphi\lrcorner
                  \Omega).
\end{align}
\end{corollary}
\begin{proof}
Let $d=\bp_{\varphi}+\p_{\varphi}$ be the corresponding $\bp$ and
$\p$ operators on $M_\varphi$. Since $\Omega\in A^{n,0}(M)$, we have
$$ (d+\partial i_\varphi-i_{\varphi}
\partial)\Omega=\bar{\partial}\Omega+\partial(\varphi\lrcorner
                  \Omega).$$ Since $e^{i_\varphi}(\Omega)\in A^{n,0}(M_\varphi)$, we have
$d\circ e^{i_\varphi}(\Omega)=\bp_\varphi (e^{i_\varphi}(\Omega))$.
Therefore, $e^{i_\varphi}(\Omega)$ is holomorphic on $M_\varphi$ if
and only if
$$\bar{\partial}\Omega+\partial(\varphi\lrcorner\Omega)=0.$$
\end{proof}

\begin{remark}
Equation (\ref{obstructionequation2}) is called the extension
equation whose solution can be used to construct the extensions of
holomorphic $(n,0)$-forms from $M$ to $M_\varphi$ which will be
discussed in detail in Section \ref{Section-extensionforms}.
\end{remark}

\section{Hodge theory and the operator $T$ on compact K\"ahler manifolds}
\label{Section-Hodgetheory}

Let $(M,\omega)$ be an $n$-dimensional compact K\"ahler manifold
with K\"ahler metric $\omega$,  and $\Vert\cdot\Vert$ be the
$L^{2}$-norm on smooth differential forms $A^{p,q}(M)$ induced by
the  metric $\omega$.
 Denote by $ L^{p,q}_2(M) $ the $L^2$-completion of
$A^{p,q}(M)$. On $A^{p,q}(M)$, we have the equality of the
Laplacians
\begin{align*}
\square_{\bar{\partial}}=\square_\partial=\frac{1}{2}\Delta_d.
\end{align*}
Let $H$ denotes the orthogonal projection from $A^{p,q}(M)$ to the
harmonic space $\mathbb{H}^{p,q}(M)=\ker \square_{\bar{\partial}}$,
We have
\begin{proposition}[cf. Pages 157-158 \cite{MK71}] \label{Prop-Hodge}
There exists a bounded operator $G$ on $A^{p,q}(M)$, called Green
operator such that
\begin{align*}
\square_{\bar{\partial}}G=G\square_{\bar{\partial}}=Id-H, \
\bar{\partial}G=G\bar{\partial}, \ \bar{\partial}^*
G=G\bar{\partial}^*, \ HG=GH=0.
\end{align*}
Moreover, $ \bar{\partial}H=H\bar{\partial}=0, \
\bar{\partial}^*H=H\bar{\partial}^*=0. $ These formulas also hold if
we replace the operator $\bp$ with $\p$.
\end{proposition}
We have the following quasi-isometry formula from \cite{LRY}.
\begin{theorem}[cf. Theorem 1.1(3)
\cite{LRY}]  \label{theorem-quasi-isometry} For any $g\in
A^{p,q}(M)$, we have
\begin{align} \label{inequ1}
\|\bar{\partial}^*G\partial g\|^2\leq \|g\|^2.
\end{align}
\end{theorem}
For reader's convenience, we provide the proof of Theorem
\ref{theorem-quasi-isometry} here.
\begin{proof}
The proof follows from the following straightforward computation
based on the Hodge theory and the formulas in Proposition
\ref{Prop-Hodge}. More precisely, we have
\begin{align}
\|\bar{\partial}^*G\partial g\|^2&=\langle \bar{\partial}^*G\partial
g,\bar{\partial}^*G\partial g\rangle
=\langle\bar{\partial}\bar{\partial}^*G\partial g, G\partial g\rangle\notag\\
&=\langle\square_{\bar{\partial}}G\partial g,G\partial g\rangle-\langle\bar{\partial}^*\bar{\partial}G\partial g, G\partial g\rangle\notag\\
&=\langle\partial g,G\partial g\rangle-\langle\bar{\partial}G\partial g,\bar{\partial}G\partial g\rangle\notag\\
&=\langle g,\square_\partial Gg\rangle-\langle g, \partial\partial^* Gg\rangle-\|\bar{\partial} G\partial g\|\notag\\
&=\langle g, g-Hg\rangle-\langle \partial^*g,G\partial^*g\rangle-\|\bar{\partial} G\partial g\|\notag\\
&=\|g\|^2-\|Hg\|^2-\langle \partial^*g,G\partial^*g\rangle-\|\bar{\partial} G\partial g\|\notag\\
&\leq \|g\|^2.\notag
\end{align}
The last inequality holds since the Green operator $G$ is a
non-negative operator.
\end{proof}

Now we consider the operator $$T=\bar{\partial}^*G\partial.$$ The
inequality $(\ref{inequ1})$ implies that $T$ is an operator of norm
less than or equal to 1 in the Hilbert space of $L^2$ forms. So we
have
\begin{corollary} \label{Tinvertible}
Given a compact K\"ahler manifold $(M,\omega)$, let $\varphi\in
A^{0,1}(M,T^{1,0}M)$ be a Beltrami differential acting on the
Hilbert space of $L^2$ forms by contraction such that its
$L_\infty$-norm $\|\varphi\|_\infty<1$, then the operator
$I+T\varphi$ is invertible on the Hilbert space of $L^2$ forms.
\end{corollary}

\section{Extension of holomorphic canonical form}
\label{Section-extensionforms} Let $(M,\omega)$ be a compact
K\"ahler manifold with $\dim_{\mathbb{C}}M=n$. From the discussions
in Section \ref{subsection-extension}, we know that, given an
integrable Beltrami differential $\varphi$ on $M$, in order to find
an $(n,0)$-form $\Omega$ on $M$ such that the corresponding
$(n,0)$-form $\rho_{\varphi}(\Omega)=e^{i_\varphi}\Omega$  is
holomorphic on $M_\varphi$, we only need to find an $(n,0)$-form
$\Omega$ on $M$ such that $\Omega$ satisfies the extension equation
\begin{equation} \label{recursionequglobal}
                  \bar{\partial}\Omega =-\partial(\varphi\lrcorner
                  \Omega).
                          \end{equation}
In this section, we show that the equation
(\ref{recursionequglobal}) can be solved by using the Hodge theory
on $(M,\omega)$ reviewed in Section \ref{Section-Hodgetheory}.

\begin{proposition} \label{prop1-section}
Let $\varphi$ be an integrable Beltrami differential of $M$ with
$L_\infty$-norm $||\varphi||_{\infty}<1$. Given a holomorphic
$(n,0)$-form $\Omega_0$ on $M$, if $\Omega$ is a solution of the
equation
\begin{align} \label{integralequation-section}
\Omega=\Omega_0-\bar{\partial}^*G\partial \left(\varphi\lrcorner
\Omega\right)= \Omega_0-T\varphi \Omega,
\end{align}
then $\Omega$ is the solution of the equation
(\ref{recursionequglobal}).
\end{proposition}
\begin{proof}
We first assume that $\Omega$ satisfies the equation $
\Omega=\Omega_0-\bar{\partial}^*G\partial \left(\varphi\lrcorner
\Omega\right).$ We need to show that
\begin{align*}
\bar{\partial}\Omega=-\partial(\varphi\lrcorner \Omega).
\end{align*}
In fact, from the formulae in Proposition \ref{Prop-Hodge}, it
follows that
\begin{align*}
\bar{\partial}\Omega&=-\bar{\partial}\bar{\partial}^*G\partial
\left(\varphi\lrcorner \Omega\right)\\\nonumber
&=(\bar{\partial}^*\bar{\partial}-\square_{\bar{\partial}})G\partial(\varphi\lrcorner
\Omega)\\\nonumber
&=(\bar{\partial}^*\bar{\partial}G-I+H)\partial(\varphi\lrcorner
\Omega)\\\nonumber &=-\partial(\varphi\lrcorner
\Omega)+\bar{\partial}^*\bar{\partial}G\partial(\varphi\lrcorner
\Omega).
\end{align*}
Let
\begin{align*}
\Phi=\bar{\partial}\Omega+\partial(\varphi\lrcorner \Omega).
\end{align*}
Then we have
\begin{align*}
\Phi&=\bar{\partial}\Omega+\partial(\varphi\lrcorner
\Omega)\\\nonumber
&=\bar{\partial}^*\bar{\partial}G\partial(\varphi\lrcorner
\Omega)\\\nonumber
&=-\bar{\partial}^*G\partial\bar{\partial}(\varphi\lrcorner
\Omega)\\\nonumber
&=-\bar{\partial}^*G\partial((\bar{\partial}\varphi)\lrcorner
\Omega+\varphi\lrcorner \bar{\partial}\Omega)\\\nonumber
&=-\bar{\partial}^*G\partial\left(\frac{1}{2}[\varphi,\varphi]\lrcorner
\Omega+\varphi\lrcorner \left(\Phi-\partial(\varphi\lrcorner
\Omega)\right)\right)\\\nonumber
&=-\bar{\partial}^*G\partial\left(\varphi\lrcorner \Phi\right),
\end{align*}
where in the last equality, we have used
(\ref{genralizedcartan_special}), $\partial \Omega=0$ and
$\partial^2=0$.

By Theorem \ref{theorem-quasi-isometry} and the condition
$||\varphi||_{\infty}<1$, we have
\begin{align}
||\Phi||^2\leq ||\phi\lrcorner \Phi||^2\leq
||\phi||_\infty||\Phi||^2< ||\Phi||^2.
\end{align}
Then we get the contradiction $\|\Phi\|^2<\|\Phi\|^2$ unless
$\Phi=0$. Hence,
\begin{align*}
\bar{\partial}\Omega=-\partial(\varphi\lrcorner \Omega).
\end{align*}
\end{proof}

Our method in the above proof of Proposition
\ref{prop1-section} is also used in \cite{LRW}, see Remark 4.6 of
\cite{LRW}.

Conversely, we have
\begin{proposition} \label{prop2-section}
If the $(n,0)$-form $\Omega$ satisfies the equation
(\ref{recursionequglobal}), then there exists a unique holomorphic
$(n,0)$-form $\Omega_0$, such that $\Omega$ satisfies the equation
(\ref{integralequation-section}).
\end{proposition}
\begin{proof}
Applying the operator $\bar{\partial}^*G$ to
(\ref{recursionequglobal}), we obtain
\begin{align*}
\bar{\partial}^*G\bar{\partial}
\Omega=-\bar{\partial}^*G\partial(\varphi\lrcorner \Omega).
\end{align*}
From the formulas in Proposition \ref{Prop-Hodge}, it follows that
\begin{align*}
\bp^*G\bp\Omega=\bp^*\bp G\Omega=\Box_{\bp}G\Omega=\Omega-H\Omega.
\end{align*}
Let $\Omega_0=H\Omega$, which is a harmonic $(n,0)$-form on $M$, it
 immediately implies that
\begin{align*}
\Omega=\Omega_0-\bar{\partial}^*G\partial \left(\varphi\lrcorner
\Omega\right).
\end{align*}
\end{proof}

Furthermore, it is easy to show that the equation
(\ref{integralequation-section}) has a unique solution. Indeed, if
we assume that the equation $(\ref{integralequation-section})$ has two
different solutions $\Omega$ and $\Omega'$, i.e. $\Omega-\Omega'\neq
0$. Then
\begin{align*}
\Omega-\Omega'=-T\varphi(\Omega-\Omega').
\end{align*}
By Theorem \ref{theorem-quasi-isometry}, we have
\begin{align*}
\|\Omega-\Omega'\|=\|T\varphi(\Omega-\Omega')\|\leq
\|\varphi(\Omega-\Omega')\|\leq \|\varphi\|_\infty
\|\Omega-\Omega'\|<\|\Omega-\Omega'\|.
\end{align*}
which contradicts to $\Omega-\Omega'\neq 0$.

By Corollary \ref{Tinvertible}, this unique solution of the equation
(\ref{integralequation-section}) is given by
\begin{align*}
\Omega=(I+T\varphi)^{-1}\Omega_0,
\end{align*}
which is a smooth $(n,0)$-form since $\Omega_0$ is holomorphic.


In conclusion, we have
\begin{theorem}\label{solution}
Given any integrable Beltrami differential $\varphi$ with
$\|\varphi\|_\infty<1$, and any holomorphic $(n,0)$-form $\Omega_0$
on $M$, we have that
\begin{align*}
\Omega(\varphi)=\rho_{\varphi}((I+T\varphi)^{-1}\Omega_0)
\end{align*}
 is a
holomorphic $(n,0)$-form on $M_{\varphi}$ with $\Omega(0)=\Omega_0$.
\end{theorem}

Applying  Theorem \ref{solution} to the integrable Beltrami
differential $\varphi(t)$ from the local Kodaira-Spencer-Kuranishi
deformaiton theory, one can choose $t$ small enough such that
$\|\varphi(t)\|_{\infty}<1$, we immediately obtain

\begin{corollary}\label{global} For any holomorphic $(n,0)$-form $\Omega_0\in
A^{n,0}(M)$, and the Beltrami differential $\varphi=\varphi(t)$ with
$|t|<\varepsilon$ small, there is a holomorphic $(n,0)$-form
$\Omega(t)$ on $M_t$,
$$\Omega(t) = \rho_t((I+T\varphi)^{-1}\Omega_0), $$
where $\rho_t=\rho_{\varphi(t)}$, with
$\Omega(0)=\Omega_0$.
\end{corollary}

 \section{Closed formulas on marked and polarized moduli spaces}\label{Section-moduli}
\subsection{The moduli space $\mathcal M$ of marked and polarized manifolds}
We refer the reader to \cite{Popp} for the basic facts on moduli
spaces in this section. Let $(M,L)$ be a polarized manifold. The
moduli space $\mathcal{M}_0$ of polarized manifolds is the complex
analytic space parameterizing the isomorphism class of polarized
manifolds with the isomorphism defined by
$$(M,L)\sim (M',L')\  \exists \ \text{biholomorphic map} \ f : M\to M' \ \text{s.t.} \ f^*L'=L .$$

We fix a lattice $\Lambda$ with a pairing $Q_{0}$, where $\Lambda$
is isomorphic to $H^n(M_{0},\mathbb{Z})/\text{Tor}$ for some $M_{0}$
in $\mathcal{M}_0$ and $Q_{0}$ is defined by the cup-product. For a
polarized manifold $(M,L)\in \mathcal{M}_0$, we define a marking
$\gamma$ as an isometry of the lattices
\begin{equation*}
\gamma :\, (\Lambda, Q_{0})\to (H^n(M,\mathbb{Z})/\text{Tor},Q)
\end{equation*}
where $Q$ is the Poincar\'e pairing.

Recall that a polarized and marked projective manifold is a triple
$(M, L, \gamma)$, where $M$ is a projective manifold, $L$ is a
polarization on $M$, and $\gamma$ is a marking $$\gamma :\,
(\Lambda, Q_{0})\to (H^n(M,\mathbb{Z})/\text{Tor},Q).$$ Two triples
$(M, L, \gamma)$ and $(M', L', \gamma')$ are called equivalent if
there exists a biholomorphic map $f:\,M\to M'$ with
\begin{align*}
f^*L'=L \ \text{and} \ f^*\gamma' =\gamma,
\end{align*}
where $f^*\gamma'$ is given by $\gamma':\, (\Lambda, Q_{0})\to
(H^n(M',\mathbb{Z})/\text{Tor},Q)$ composed with
$$f^*:\, (H^n(M',\mathbb{Z})/\text{Tor},Q)\to (H^n(M,\mathbb{Z})/\text{Tor},Q).$$
We denote by $[M, L, \gamma]$ the isomorphism class of polarized and
marked projective manifolds of $(M, L, \gamma)$.

The moduli space $\mathcal M$ of marked and polarized manifolds is
the complex analytic space parameterizing the isomorphism class of
marked and polarized manifolds.
Let $X$ denote the underlying real manifold for $M$. From the
definition we know that the first Chern class $c_1(L)\in H^2(X)$ is
fixed, which we can take to be the K\"ahler class on $M$. We will
only consider the connected component of the marked moduli space
containing $M$ which we still denote by $\mathcal{M}$.

\subsection{Closed explicit formula for the global section of holomorphic forms}
Now we take two distinct points in ${\mathcal{M}}$ whose
corresponding fibers are $M_0$ and $M_1$. Let $X$ be the background
smooth manifold. Denote the corresponding K\"ahler form on $M_0$ and
$M_1$ by $\omega_0$ and $\omega_1$ respectively, which we consider
as symplectic forms on $X$. Since
$$c_1(L) = [\omega_0]=[\omega_1]\in H^2(X),$$
by Theorem 2  of Moser in \cite{Moser}, we know that there exists a
continuous  family of diffeomorphisms $f_t$ of $X$ with $t\in [0,1]$
and $f_0=\mathrm{id}$, such that $\omega_t= f_t^* \omega_0$ and
$f_1^*\omega_0 =\omega_1$. Therefore by pulling back the complex
structure and $\omega_1$ using $f_t^{-1}$,  we may assume that,
considered as a symplectic form, $M_1$ also has K\"ahler form
$\omega_0$.
\begin{theorem} \label{Theorem-closedsection}
Given a point $(M,\omega_0)$ in the marked and polarized moduli
space $\mathcal{M}$ and a holomorphic $n$-form $s_0$ on $M$, there
is a canonical section $s$ of the Hodge bundle $\mathcal{H}^{n,0}$,
such that for any point $M_1$ in $\mathcal{M}$, the de Rham
cohomology class of $s$ in $H^*(M_1)$ is represented by
$$s = \rho_\varphi((I+T\phi)^{-1} s_0).$$
where $\phi$ is the Beltrami differential associated to $M_1$, and
$T=\bar{\partial}^*G\partial $ is the operator of the Hodge theory
on $M$ with K\"ahler metric $\omega_0$.
\end{theorem}
\begin{proof}
First note that on the moduli space with markings, the de Rham
cohomology groups $H^n(M_1)$ is canonically identified to $H^n(M)$
which is independent of the point $M_1$ in $\mathcal{M}$. Second,
since we will construct the de Rham cohomology classes in $H^n(M)$
of the holomorphic forms, and the de Rham cohomology class is
independent of continuous diffeomorphisms, therefore by the above
discussion using Moser's theorem in \cite{Moser}, we may assume that
$M$ and $M_1$ have the same symepletic form $\omega_0$.  By the
following Lemma \ref{form}, we know that associated to the complex
structure on $M_1$, there exists the Beltrami differential $\phi\in
A^{0,1}(M,T^{1,0}M)$, with $||\phi||_\infty<1$, where the norm is
taken with respect to the metric $\omega_0$ on $M$.

From Corollary \ref{solution}, since $||\phi||_\infty<1$ on $M$, we
deduce that given the holomorphic $n$-form $s_0$, there exists the
holomorphic $n$-form
$$s = \rho_\varphi((I+T\phi)^{-1}s_0)$$ on $M_1$.
Note that, considered as de Rham cohomology class, the above formula
is independent of the continuous diffeomorphisms used in the Moser
theorem to identify the K\"ahler form $\omega_1$ on $M_1$ to the
K\"ahler form $\omega_0$ on $M$, both considered as symplectic
forms. Therefore, as de Rham cohomology classes, the formula holds
on the marked and polarized moduli space $\mathcal{M}$.
\end{proof}

\begin{lemma}\label{form}
\label{trans} Let $M_0$, $M_1$ be any two points in the moduli space
$\mathcal{M}$. If $M_0$ and $M_1$ have the same K\"ahler form
$\omega$, then the integrable almost complex structure on
$M_1$ is of finite distance from $M_0$. More precisely, there exists a unique Beltrami differential $%
\phi\in A^{0,1}\left ( M_0,T^{1,0}M_0\right )$ such that the complex
structure on $M_1=(M_0)_\varphi$. Moreover the $L_\infty$-norm of
$\varphi$, $||\varphi||_{\infty}<1$.
\end{lemma}

\begin{proof}
Since $ T^{1,0}M_0\oplus T^{0,1}M_0=T_{\mathbb{C}}X=T^{1,0}M_1\oplus
T^{0,1}M_1, $ we let
\begin{align*} \iota_1: T_{\mathbb{C}}X\rightarrow
T^{1,0}M_0, \  \iota_2: T_{\mathbb{C}}X\rightarrow T^{0,1}M_0,
\end{align*}
be the corresponding projection maps.

We let $(M,\omega)$ be any K\"ahler manifold, let $g$ be the
corresponding K\"ahler metric, let $J$ be the complex structure and
let $p\in M$ be a point. Then for any vector $v\in T_p^{1,0}M$ such
that $v\ne 0$, we know that
\begin{equation*}
0<\Vert v\Vert^2=g(v,\bar v)=\omega(v,J(\bar
v))=-\sqrt{-1}\omega(v,\bar v).
\end{equation*}
Namely, $-\sqrt{-1}\omega(v,\bar v)>0$. Since $\omega$ is
skew-symmetric, we have that
$$-\sqrt{-1}\omega(u,\bar u)<0$$ for any
nonzero vector $u\in T_p^{0,1}M$.

Now we pick any $v\in T_p^{1,0}M_1$ such that $v\ne 0$. By the above
argument we have
\begin{align} \label{inequv}
-\sqrt{-1}\omega(v,\bar v)>0.
\end{align}
Let $v_1=\iota_1(v)\in T_p^{1,0}M_0$ and $v_2=\iota_2(v)\in
T_p^{0,1}M_0$. It follows from type considerations that
$$\omega(v_1,\bar v_2)=0=\omega(v_2,\bar v_1).$$ If $v_1=0$ then
$$-\sqrt{-1}\omega(v,\bar v)= -\sqrt{-1}\omega(v_2,\bar v_2)<0$$ which
is
a contradiction. Thus we know that $v_1\ne 0$ which implies that $%
\iota_1\mid_{T_p^{1,0}M_1}$ is a linear isomorphism. Thus,  the
integrable almost complex structure $M_1$ is of finite distance from
$M_0$. By the discussion in Section \ref{subsection-Beltrami}, the
integrable almost complex structure $M_1$ gives an integrable
Beltrami differential $\varphi \in A^{1,0}(M_0,T^{1,0}M_0)$
determined by the map $$\bar{\varphi}(v)=-\iota_2\circ
\left(\iota_1|_{T_p^{1,0}M_1}\right)^{-1}(v)$$ for $v\in
T_p^{1,0}M_0$, such that $M_1=(M_0)_\varphi$.

Moreover, since $v=v_1+v_2$, the inequality (\ref{inequv}) implies
\begin{equation}\label{norm}
\Vert v_1\Vert_0>\Vert v_2\Vert_0
\end{equation}
where the norm is measured with respect to the K\"ahler metric on
$M_0$. This inequality holds at each point in $X$.

Then the inequality (\ref{norm}) tells us that the norm of
$\bar\varphi$ is pointwise less than $1$, which implies that the
$L_\infty$-norm of $\varphi$,
$$||\varphi||_{\infty}=||\bar{\varphi}||_{\infty}<1.$$
\end{proof}

\begin{remark}
Lemma \ref{form} was first discovered in a joint project of the
first author with X. Sun, A. Todorov and S.-T. Yau. The closed
formula for holomorphic forms in the above Theorem
\ref{Theorem-closedsection} is useful to study the global Torelli
problems about injectivity of period maps.
\end{remark}

\section{Extension of pluricanonical forms}
\label{Section-pluricanonical} In Section
\ref{Section-extensionforms}, we have presented a closed explicit extension
formula for the holomorphic $(n,0)$-form (i.e. canonical form) over
compact K\"ahler manifolds. In this section, we generalize the
method used in Section \ref{Section-extensionforms} to construct a
closed explicit extension formula for the pluricanonical form. Actually, it
 is closely related to Siu's conjecture of the invariance of plurigenera
 for compact K\"ahler manifolds \cite{Siu}.

The following Sections \ref{subsection-extensionequations},
\ref{subsection-quasiisometry}, \ref{subsection-solvingequation} can
be read by comparing with Sections \ref{Section-extensionequation},
\ref{Section-Hodgetheory}, \ref{Section-extensionforms}
respectively. In Section \ref{subsection-KEextensionformula}, we
present a closed explicit formula for the extension of pluricanonical forms
over the K\"ahler-Einstein manifolds of general type.

\subsection{Extension equations} \label{subsection-extensionequations}
Let $(M,\omega)$ be a compact K\"ahler manifold of dimension $n$,
let $\varphi\in A^{0,1}(M,T^{1,0}M)$ be an integrable Beltrami
differential.  $\varphi$ determines a new complex manifold denoted
by $M_\varphi$. Given an integer $m\geq 2$, we introduce the line
bundles $\mathcal{L}_M=K_{M}^{\otimes (m-1)}$ and
$\mathcal{L}_{M_\varphi}=K_{M_\varphi}^{\otimes (m-1)}$. For a
$\mathcal{L}_M$-valued $(n,0)$-form $\sigma$ on $M$, we can deform
it via the complex structures $\varphi$. We define the map
\begin{align*}
\rho_{\varphi}: A^{n,0}(M,\mathcal{L}_{M})\rightarrow
A^{n,0}(M_\varphi,\mathcal{L}_{M_\varphi})
\end{align*}
as follows. For any $x\in M$, one can pick a local holomorphic
coordinate system $\{z^1,...,z^n\}$ near $x$, we write the
integrable Beltrami differential $\varphi$ as
\begin{align*}
\varphi=\varphi_{\bar{k}}^{i}d\bar{z}^{k}\otimes \p_i.
\end{align*}
Let $\sigma=f(z)dz^1\wedge \cdots \wedge dz^n\otimes e$, where
$e=(dz^1\wedge \cdots \wedge dz^n)^{m-1}$, we define
\begin{align} \label{rho}
\rho_\varphi(\sigma)=f(z)((dz^1+\varphi(dz^1))\wedge \cdots \wedge
(dz^n+\varphi(dz^n)) )^{\otimes m}.
\end{align}

Let $\{w^1,...,w^n\}$ be a local holomorphic coordinate system of
$M_\varphi$. Then
\begin{align*}
dw^i=\p_j w^i dz^{j}+\p_{\bar{j}} w^i d\bar{z}^{j}=\p_j
w^i(dz^{j}+\varphi(dz^j))
\end{align*}
by the definition of the Beltrami differential $\varphi$. Indeed, if
we let $a=(a_{ij})=(\p_j w^i)$ and $a^{-1}=(a^{ij})$, then
\begin{align*}
\varphi_{\bar{k}}^{i}=a^{ij}\p_{\bar{k}} w^j, \
\varphi=\varphi_{\bar{k}}^{i}d\bar{z}^{k}\otimes \p_i.
\end{align*}
Hence, the formula (\ref{rho}) can be rewritten  as
\begin{align*}
\rho_\varphi(\sigma)=\frac{f(z)}{\det(a)^m}(dw^1\wedge \cdots
dw^n)^{\otimes m}.
\end{align*}

\begin{lemma} \label{lemmalocalequ}
Given $\sigma=f(z)dz^1\wedge \cdots dz^n\otimes e\in
A^{n,0}(M,\mathcal{L}_M)$, then $\rho_{\varphi}(\sigma)$ is
holomorphic in $A^{n,0}(M_\varphi,\mathcal{L}_{M_\varphi})$ if and
only if for $j=1,...,n$,
\begin{align} \label{rhoholomorphy}
\bp_j f=\varphi_{\bar{j}}^{i}\p_i f+mf\p_i \varphi_{\bar{j}}^{i}.
\end{align}
\end{lemma}
\begin{proof}
Since a local smooth function $h$ is holomorphic on $M_\varphi$ if
and only if
\begin{align*}
\bp h=\varphi\lrcorner (\p h),
\end{align*}
i.e. for $j=1,...,n,$
\begin{align*}
\bp_{j} h=\varphi_{\bar{j}}^{i}\p_i h.
\end{align*}

Therefore, $\rho_{\varphi}(\sigma)$ is holomorphic, i.e.
$\frac{f(z)}{\det(a)^m}$ is holomorphic on $M_\varphi$, if and only
if
\begin{align*}
\bp_{j}
\left(\frac{f(z)}{\det(a)^m}\right)=\varphi_{\bar{j}}^{i}\p_i\left(
\frac{f(z)}{\det(a)^m}\right)
\end{align*}
which is equivalent to
\begin{align} \label{rhoholom}
\left(\bp_j f-mf a^{ik}\bp_j
a_{ki}\right)=\left(\varphi_{\bar{j}}^{i}\p_i
f-mf\varphi_{\bar{j}}^{i}a^{pl}\p_{i}a_{lp}\right)
\end{align}
by a straightforward computation.

We claim that
\begin{align} \label{claimidentity}
a^{ik}\bp_ja_{ki}-\varphi_{\bar{j}}^{i}a^{pl}\p_{i}a_{lp}=\p_{i}\varphi_{\bar{j}}^{i}.
\end{align}
In fact, we have
\begin{align*}
\p_{i}\varphi_{\bar{j}}^{i}&=\p_i(a^{ik}\bp_{j}w^k)=\p_ia^{ik}\bp_{j}w^k+a^{ik}\p_i\bp_j
w^k\\\nonumber
&=-a^{ip}\p_ia_{pl}a^{lk}\bp_{j}w^k+a^{ik}\bp_{j}a_{ki}\\\nonumber
&=-a^{ip}\p_{l}a_{pi}\varphi_{\bar{j}}^{l}+a^{ik}\bp_{j}a_{ki}
\end{align*}
which gives (\ref{claimidentity}).
 Therefore, substituting (\ref{claimidentity}) into (\ref{rhoholom}),
 we obtain (\ref{rhoholomorphy}).
\end{proof}

Let $D=D'+\bp$ be the Chern connection of the holomorphic bundle
$T^{1,0}M$ over $M$. The connection matrix is given by $\theta=(\p g
g^{-1})$, where $g=(g_{i\bar{j}})$ denotes the K\"ahler metric
matrix associated to the K\"ahler form $\omega$. We define the
divergence operator $div$ as $tr\circ D'$. For
$\varphi=\varphi_{\bar{j}}^{i}d\bar{z}^j\otimes \p_i\in
A^{0,1}(M,T^{1,0}M)$, we have
\begin{align*}
D'\varphi&=\p(\varphi_{\bar{j}}^{i}d\bar{z}^j)\p_i-
\varphi_{\bar{j}}^{i}d\bar{z}^j (\p g g^{-1})_i^{p}\p_{p}\\\nonumber
&=\p_k\varphi_{\bar{j}}^{i}dz^k\wedge
d\bar{z}^j\p_i-\varphi^{i}_{\bar{j}}d\bar{z}^j\p_k
g_{i\bar{l}}g^{\bar{l}p}dz^k\p_p.
\end{align*}
Therefore
\begin{align*}
div \varphi=tr\circ
D'(\varphi)&=\p_i\varphi_{\bar{j}}^id\bar{z}^j+\varphi^{i}_{\bar{j}}\p_k
g_{i\bar{l}}g^{\bar{l}k}d\bar{z}^j\\\nonumber
&=(\p_i\varphi_{\bar{j}}^i+\varphi^{i}_{\bar{j}}\p_i
g_{k\bar{l}}g^{\bar{l}k})d\bar{z}^j\\\nonumber
&=(\p_i\varphi_{\bar{j}}^i+\varphi^{i}_{\bar{j}}\p_i \log \det(g)
)d\bar{z}^j
\end{align*}
where we have used the K\"ahler condition $\p_k g_{i\bar{l}}=\p_i
g_{k\bar{l}}$.

Let $\nabla'$ be the $(1,0)$-component of the naturally induced
Chern connection on the holomorphic line bundle
$\mathcal{L}_M=K_{M}^{\otimes(m-1)}$.  The induced Hermitian metric
on $\mathcal{L}_M$ is given by $(\det g)^{-(m-1)}$. For a
holomorphic section $e$ of $\mathcal{L}_M$, we have
\begin{align*}
\nabla'e&=\partial\left((\det g)^{-(m-1)}\right)(\det
g)^{(m-1)}e\\\nonumber &=-(m-1)\p_i \log(\det g)dz^i\otimes e.
\end{align*}

\begin{proposition} \label{lemmaglobalequ}
Given $\sigma\in A^{n,0}(M,\mathcal{L}_M)$, then
$\rho_{\varphi}(\sigma)$ is holomorphic in
$A^{n,0}(M_\varphi,\mathcal{L}_{M_\varphi})$ if and only if
\begin{align} \label{extensionequ}
\bp \sigma=-\nabla'(\varphi \lrcorner \sigma)+(m-1)div \varphi\wedge
\sigma.
\end{align}
\end{proposition}
\begin{proof}
Let $\sigma=fdz^1\wedge \cdots \wedge dz^n\otimes e\in
A^{n,0}(M_0,\mathcal{L}_{M_0})$, then
\begin{align*}
\varphi\lrcorner \sigma=(-1)^{n+i}f\varphi_{\bar{j}}^{i}dz^1\wedge
\cdots \wedge \widehat{dz^i}\wedge \cdots\wedge dz^n\wedge
d\bar{z}^j\otimes e,
\end{align*}
by a straightforward computation. We have
\begin{align*}
\nabla'(\varphi\lrcorner \sigma)&=
(-1)^{n+i}\p(f\varphi_{\bar{j}}^{i})dz^1\wedge \cdots \wedge
\widehat{dz^i}\wedge \cdots\wedge dz^n\wedge d\bar{z}^j\otimes
e\\\nonumber &+(-1)^if\varphi_{\bar{j}}^{i}dz^1\wedge \cdots \wedge
\widehat{dz^i}\wedge \cdots\wedge dz^n\wedge d\bar{z}^j\wedge
\nabla'e\\\nonumber &=-\p_i(f\varphi_{\bar{j}}^{i})d\bar{z}^j\wedge
dz^1\wedge \cdots \wedge  dz^n\otimes e\\\nonumber
&+(m-1)f\varphi_{\bar{j}}^{i}\p_i \log(\det g)d\bar{z}^j\wedge
dz^1\wedge \cdots \wedge  dz^n\otimes e.
\end{align*}
We also have
\begin{align*}
(m-1)div \varphi\wedge \sigma=(m-1)f\left(\p_i
\varphi_{\bar{j}}^i+\varphi_{\bar{j}}^{i}\p_i \log (\det
g)\right)d\bar{z}^j\wedge dz^1\wedge \cdots \wedge  dz^n\otimes e,
\end{align*}
\begin{align*}
\bp \sigma=(\bp_j f)d\bar{z}^j\wedge dz^1\wedge \cdots \wedge
dz^n\otimes e.
\end{align*}
Therefore, identity (\ref{extensionequ}) follows from Lemma
\ref{lemmalocalequ}.
\end{proof}

\begin{remark}
Note that $\sigma\in A^{n,0}(M,\mathcal{L}_M)$ can also be regarded
as a smooth section of the holomorphic line bundle $K_{M}^{\otimes m}$
since $ A^{n,0}(M,\mathcal{L}_M)=A^{0,0}(M,K_{M}^{\otimes m}). $
 The equation
(\ref{extensionequ}) is called the extension equation of the
pluricanonical form, which gives the criteria when the extended
pluricanonical form is holomorphic under the new complex structure.

On the other hand side, if we let $\widehat{\nabla}'$ be the
$(1,0)$-part of the Chern connection on $K_{M}^{\otimes m}$, it is
easy to see that the extension equation (\ref{extensionequ}) is
equivalent to the equation
\begin{align*}
\bp \sigma=\varphi\lrcorner\widehat{\nabla}'\sigma+m \,div
\varphi\wedge \sigma
\end{align*}
which  was first derived in a joint project of the first author with
X. Sun, A. Todorov and S.-T. Yau \cite{LSTY}.
\end{remark}

\subsection{Bundle-valued quasi-isometry over compact K\"ahler
manifold} \label{subsection-quasiisometry} In this section, we
review the bundle-valued quasi-isometry formula obtained in
\cite{LRY}. Let $(E,h)$ be a Hermitian holomorphic vector bundle
over the compact K\"ahler manifold $(M,\omega)$ and
$\nabla=\nabla'+\bp$ be the Chern connection of $(E,h)$. With
respect to metrics on $E$ and $X$, we set
\begin{align*}
\overline{\square}&=\bp \bp^*+\bp^*\bp, \\
\square'&=\nabla'\nabla'^*+\nabla'^*\nabla'.
\end{align*}

Accordingly, we have the Green operator $G$ (resp. $G'$) and
harmonic projection $H$ (resp. $H'$) in the Hodge decomposition corresponding to
$\square$ (resp. $\overline{\square}$). We have
\begin{align*}
I=H+\overline{\square}\circ G, \ \ I=H'+\overline{\square}'\circ G'.
\end{align*}
Let $\{z^i\}_{i=1}^{n}$ be the local holomorphic coordinates on $M$
and $\{e_\alpha\}_{\alpha=1}^{r}$ be a local frame of $E$.  Let
$h=(h_{\alpha\bar{\beta}})$ where
$h_{\alpha\bar{\beta}}=h(e_\alpha,e_\beta)$, and the inverse matrix
$h^{-1}=(h^{\bar{\alpha}\beta})$. By the curvature formula of Chern
connection $\Theta=\bp (\p h h^{-1})$, we obtain
\begin{align*}
\Theta_{i\bar{j}\alpha}^{\delta}=-\left(\frac{\p^2
h_{\alpha\bar{\beta}}}{\p z^i \p
\bar{z}^j}\right)h^{\bar{\beta}\delta}-\frac{\p
h_{\alpha\bar{\beta}}}{\p z^i}\frac{\p h^{\bar{\beta}\delta}}{\p
\bar{z}^j}.
\end{align*}
Let $R_{i\bar{j}}=\sum_{\alpha=1}^{r}
\Theta_{i\bar{j}\alpha}^{\alpha}$, we define the Chern-Ricci form of
$(E,h)$ by
\begin{align*}
Ric(E,h)=\frac{\sqrt{-1}}{2}R_{i\bar{j}}dz^{i}\wedge d\bar{z}^{j}.
\end{align*}
In particular, when $E=T^{1,0}M$, the corresponding Chern-Ricci form
 is given by
\begin{align*}
Ric(\omega)=\frac{\sqrt{-1}}{2}\bp\p\log(\det g).
\end{align*}
Let
$\Theta_{i\bar{j}\alpha\bar{\beta}}=\Theta_{i\bar{j}\alpha}^{\gamma}h_{\gamma\bar{\beta}}$,
we obtain
\begin{align*}
\Theta_{i\bar{j}\alpha\bar{\beta}}=-\frac{\p^2
h_{\alpha\bar{\beta}}}{\p z^i \p
\bar{z}^j}+h^{\bar{\delta}\gamma}\frac{\p h_{\alpha\bar{\delta}}}{\p
z^i}\frac{\p h_{\gamma\bar{\beta}}}{\p \bar{z}^j}.
\end{align*}
\begin{definition}
An Hermitian vector bundle $(E,h)$ is said to be semi-Nakano positive
(resp. Nakano-positive), if for any non-zero vector
$u=u^{i\alpha}\p_i\otimes e_\alpha$,
\begin{align*}
\sum_{i,j,\alpha,\beta}\Theta_{i\bar{j}\alpha\bar{\beta}}u^{i\alpha}\bar{u}^{j\beta}\geq
0, \ (\text{resp.} \ >0).
\end{align*}
In particular, for a line bundle, we say that it is positive, if it is
Nakano-positive.
\end{definition}

\begin{proposition}[cf. Theorem 1.1(2) in \cite{LRY}] \label{Theorem-bundlevalue-quasi-isometry}
If $(\mathcal{L},h)$ is a positive line bundle over a compact
K\"ahler manifold $(M,\omega)$ and $\sqrt{-1}\Theta=\rho\omega$ for
a constant $\rho>0$, then for any $f\in
A^{n-1,\bullet}(M,\mathcal{L})$, we have
\begin{align} \label{quasi-isometry-vectorvalue}
\|\bp^*G\nabla' f\|\leq \|f\|.
\end{align}
\end{proposition}
For reader's convenience, we provide the proof of Proposition
\ref{Theorem-bundlevalue-quasi-isometry} here.
\begin{proof}
By the well-known Bochner-Kodaira-Nakano identity
\begin{align*}
\overline{\square}=\square'+[\sqrt{-1}\Theta,\Lambda_{\omega}],
\end{align*}
and  $[\omega,\Lambda_\omega]=(k-n)I$ on $A^{k}(M)$, we have
\begin{align*}
\overline{\square}(\nabla'f)=\square'(\nabla'f)+\rho
q(\nabla'f)=(\square'+\rho q)(\nabla'f),
\end{align*}
for any $f\in A^{n-1,q}(M,\mathcal{L})$. Hence
$$\mbox{Ker}\, \overline{\square}\subseteq \mbox{Ker}\,\square'$$ which implies that
$H\nabla' f=0$. Thus
\begin{align*}
\overline{\square}G(\nabla' f)=\nabla'f=\square'G'(\nabla' f)
\end{align*}
by $H'(\nabla' f)=0$ and the Hodge decomposition for $\nabla' f$.
Then
\begin{align*}
\langle \nabla' f,G(\nabla' f) \rangle&=\langle \nabla' f,
\overline{\square}^{-1}(\nabla' f) \rangle\\\nonumber &=\langle
\nabla' f, (\square'+\rho q)^{-1}(\nabla' f) \rangle\\\nonumber
&\leq \langle \nabla' f, \square'^{-1}(\nabla' f) \rangle\\\nonumber
&=\langle \nabla' f, G'(\nabla' f)\rangle.
\end{align*}
Therefore,
\begin{align*}
\|\bp^*G\nabla'f\|^2&=\langle \bp^*G\nabla' f, \bp^*G(\nabla'
f)\rangle\\\nonumber &=\langle G\nabla' f, \bp\bp^*G(\nabla'
f)\rangle\\\nonumber &=\langle G\nabla' f,
(\overline{\square}-\bp^*\bp)G(\nabla' f)\rangle\\\nonumber
&=\langle G\nabla' f, \nabla' f\rangle-\langle\bp G\nabla' f,\bp
G\nabla' f \rangle\\\nonumber &\leq\langle \nabla' f, G(\nabla'
f)\rangle\\\nonumber &\leq\langle \nabla' f, G'(\nabla'
f)\rangle\\\nonumber&=\langle  f, \nabla'^*\nabla'G'
f\rangle\rangle\\\nonumber&=\langle
f,f-H'(f)-\nabla'{\nabla'}^*G'f\rangle\\\nonumber&=\|f\|^2-\|H'(f)\|^2
-\langle\nabla'^*f,G'\nabla'^*f\rangle\\\nonumber& \leq\|f\|^2.
\end{align*}
\end{proof}

We introduce the operator
\begin{align*}
T^{\nabla'}=\bp^*G\nabla'.
\end{align*}

The quasi-isometry formula (\ref{quasi-isometry-vectorvalue}) implies that
$T^{\nabla'}$ is an operator of norm less than or equal to $1$ in
the $L^2$ Hilbert space of the $\mathcal{L}$-valued forms. So we
have
\begin{corollary} \label{Coro-inverse}
Let $(\mathcal{L},h)$ be a positive line bundle over a compact
K\"ahler manifold $(M,\omega)$, with $\sqrt{-1}\Theta=\rho\omega$
for a constant $\rho>0$. Let $\varphi\in A^{0,1}(M,T^{1,0}M)$ be a
Beltrami differential acting on the Hilbert space
$L_{2}^{n,\bullet}(X,\mathcal{L})$  by contraction such that its
$L_\infty$-norm $\|\varphi\|_\infty<1$. Then the operator
$I+T^{\nabla'}\varphi$ is invertible.
\end{corollary}

\begin{example}
We will consider the holomorphic line bundle
$\mathcal{L}_M=K_{M}^{\otimes (m-1)}$ over the compact K\"ahler
manifold $(M,\omega)$, the corresponding Hermitian metric is given
by $h_\omega=(\det g)^{-{(m-1)}}$. In this case, the curvature of
the Chern connection of $\mathcal{L}_M$ is given by
\begin{align*}
\Theta=-(m-1)\bp\p \log (\det g).
\end{align*}
Therefore
\begin{align} \label{curvature-linebundle}
\sqrt{-1}\Theta=-2(m-1)Ric(\omega).
\end{align}
In particular, if $(M,\omega)$ is a K\"ahler-Einstein manifold of
general type as defined in Definition \ref{def-KEgeneraltype}, i.e.
$Ric(\omega)=-\omega$, then we have
\begin{align*}
\sqrt{-1}\Theta=2(m-1)\omega.
\end{align*}
\end{example}

\subsection{Solving the extension equation}
\label{subsection-solvingequation} As discussed in Section
\ref{subsection-extensionequations}, in order to construct an
extension pluricanonical form over $M_\varphi$, we need to solve the
extension equation (\ref{extensionequ}).

Before going further, we need the following lemma.
\begin{lemma} \label{keylemma}
Let $\varphi\in A^{0,1}(M,T^{1,0}M)$ be an integrable Beltrami
differential and let $\sigma\in A^{n,0}(M,\mathcal{L}_M)$, we set
\begin{align}
\Psi=\bp \sigma+\nabla'(\varphi\lrcorner \sigma)-(m-1)div\varphi
\wedge \sigma,
\end{align}
then we have the identity:
\begin{align} \label{identity}
\bp(\nabla'(\varphi\lrcorner \sigma)-(m-1)div\varphi \wedge
\sigma)=-\left(\nabla'(\varphi\lrcorner \Psi)-(m-1)div\varphi\wedge
\Psi\right).
\end{align}
\end{lemma}
\begin{proof}
Locally, we write $\varphi=\varphi_{\bar{j}}^{i}d\bar{z}^j\otimes
\p_i\in A^{0,1}(M,T^{1,0}M)$, $\sigma=f dz^1\wedge \cdots\wedge
dz^n\otimes e \in A^{n,0}(M,\mathcal{L}_M)$ where $e=(dz^1\wedge
\cdots\wedge dz^n)^{\otimes (m-1)}$. Then
\begin{align*}
div\varphi=(\p_i\varphi_{\bar{j}}^i+\varphi^{i}_{\bar{j}}\p_i \log
\det(g) )d\bar{z}^j.
\end{align*}
For brevity, we introduce the notations $dZ=dz^1\wedge \cdots\wedge
dz^n$ and $dZ^{[k]}=dz^1\wedge \cdots\wedge
\widehat{dz^k}\wedge\cdots \wedge dz^n$, where the hat indicates
that the corresponding term is to be dropped.

By the computations in the proof of Proposition
\ref{lemmaglobalequ}, we have
\begin{align*}
\nabla'(\varphi\lrcorner \sigma)-(m-1)div\varphi \wedge
\sigma=-\left((\p_i
f)\varphi_{\bar{j}}^i+mf\p_i\varphi_{\bar{j}}^i\right)d\bar{z}^j\wedge
dZ\otimes e.
\end{align*}
Therefore
\begin{align} \label{identity-LHS}
&\bp(\nabla'(\varphi\lrcorner \sigma)-(m-1)div\varphi \wedge
\sigma)\\\nonumber &=-\bp_l\left((\p_i
f)\varphi_{\bar{j}}^i+mf\p_i\varphi_{\bar{j}}^i\right)d\bar{z}^l\wedge
d\bar{z}^j\wedge dZ\otimes e\\\nonumber &=\sum_{1\leq l<j\leq
n}\left((\bp_j\p_if
\varphi_{\bar{l}}^i-\bp_l\p_if\varphi_{\bar{j}}^i)+\p_if(\bp_j\varphi_{\bar{l}}^i-\bp_l
\varphi_{\bar{j}}^i)\right.\\\nonumber &\left.+m(\bp_j
f\p_i\varphi_{\bar{l}}^i-\bp_l
f\p_i\varphi_{\bar{j}}^i)+mf(\bp_j\p_i\varphi_{\bar{l}}^i-\bp_l\p_i\varphi_{\bar{j}}^i)\right)d\bar{z}^l\wedge
d\bar{z}^j\wedge dZ\otimes e.
\end{align}

On the other hand side, since
\begin{align*}
\Psi&=\bp \sigma+\nabla'(\varphi\lrcorner \sigma)-(m-1)div\varphi
\wedge \sigma\\\nonumber &=\left(\bp_j f-\varphi_{\bar{j}}^i\p_i
f-mf\p_i\varphi_{\bar{j}}^i\right)d\bar{z}^j\wedge dZ\otimes e,
\end{align*}
we have
\begin{align*}
\varphi\lrcorner \Psi&=(\varphi_{\bar{l}}^{k}d\bar{z}^l\otimes
\p_k)\lrcorner\left(\bp_j f-\varphi_{\bar{j}}^i\p_i
f-mf\p_i\varphi_{\bar{j}}^i\right)d\bar{z}^j\wedge dZ\otimes
e\\\nonumber
&=\sum_{k=1}^n(-1)^k\varphi_{\bar{l}}^k(\bp_jf-\varphi_{\bar{j}}^i\p_if-mf\p_i\varphi_{\bar{j}}^i)d\bar{z}^l\wedge
d\bar{z}^j\wedge dZ^{[k]}\otimes e.
\end{align*}
Thus
\begin{align*}
&\nabla'(\varphi\lrcorner \Psi)\\\nonumber
&=\p\left(\sum_{k=1}^n(-1)^k\varphi_{\bar{l}}^k(\bp_jf-\varphi_{\bar{j}}^i\p_if-mf\p_i\varphi_{\bar{j}}^i)\right)d\bar{z}^l\wedge
d\bar{z}^j\wedge dZ^{[k]}\otimes e\\\nonumber
&+(-1)^{n+1}\sum_{k=1}^n(-1)^k\varphi_{\bar{l}}^k(\bp_jf-\varphi_{\bar{j}}^i\p_if-mf\p_i\varphi_{\bar{j}}^i)d\bar{z}^l\wedge
d\bar{z}^j\wedge dZ^{[k]}\wedge \nabla' e\\\nonumber
&=-\p_k\left(\varphi_{\bar{l}}^k(\bp_jf-\varphi_{\bar{j}}^i\p_if-mf\p_i\varphi_{\bar{j}}^i)\right)d\bar{z}^l\wedge
d\bar{z}^j\wedge dZ\otimes e\\\nonumber
&+\varphi_{\bar{l}}^k\left(\bp_jf-\varphi_{\bar{j}}^i\p_if-mf\p_i\varphi_{\bar{j}}^i\right)(m-1)\p_k\log
(\det g)d\bar{z}^l\wedge d\bar{z}^j\wedge dZ\otimes e.
\end{align*}
We also have
\begin{align*}
-(m-1)div \varphi\wedge \Psi
&=-(m-1)\left(\p_i\varphi_{\bar{l}}^i+\varphi^{i}_{\bar{l}}\p_i \log
\det(g) \right)\\\nonumber
&\cdot\left(\bp_jf-\varphi_{\bar{j}}^i\p_if-mf\p_i\varphi_{\bar{j}}^i\right)d\bar{z}^l\wedge
d\bar{z}^j\wedge dZ\otimes e.
\end{align*}
Therefore
\begin{align} \label{identity-RHS}
&-\left(\nabla'(\varphi\lrcorner \Psi)-(m-1)div\varphi\wedge
\Psi\right)\\\nonumber &=m(\p_k
\varphi_{\bar{l}}^k)\left(\bp_jf-\varphi_{\bar{j}}^i\p_if-mf\p_i\varphi_{\bar{j}}^i\right)
+\varphi_{\bar{l}}^k\p_k\left(\bp_jf-\varphi_{\bar{j}}^i\p_if-mf\p_i\varphi_{\bar{j}}^i\right)\\\nonumber
&\cdot d\bar{z}^l\wedge d\bar{z}^j\wedge dZ\otimes e\\\nonumber
&=\sum_{0\leq l\leq j\leq
n}\left(m\left(\p_k\varphi_{\bar{l}}^k\bp_jf-\p_k\varphi_{\bar{j}}^k\bp_lf\right)+
\left(\varphi_{\bar{l}}^k\p_k\bp_jf-\varphi_{\bar{j}}^k\p_k\bp_lf\right)\right.\\\nonumber
&\left.
+\left(\varphi_{\bar{j}}^k\p_k\varphi_{\bar{l}}^i\p_if-\varphi_{\bar{l}}^k\p_k\varphi_{\bar{j}}^i\p_if\right)
+mf\left(\varphi_{\bar{j}}^k\p_k\p_i\varphi_{\bar{l}}^i-\varphi_{\bar{l}}^k\p_k\p_i\varphi_{\bar{j}}^i\right)\right)\\\nonumber
&\cdot d\bar{z}^l\wedge d\bar{z}^j\wedge dZ\otimes e.
\end{align}

Since $\varphi$ is integrable, i.e.
$\bp\varphi=\frac{1}{2}[\varphi,\varphi]$, we obtain
\begin{align*}
\varphi_{\bar{l}}^k\p_k\varphi_{\bar{j}}^i-\varphi_{\bar{j}}^k\p_k\varphi_{\bar{l}}^i=\bp_l\varphi_{\bar{j}}^i-\bp_j\varphi_{\bar{l}}^i,
\end{align*}
and
\begin{align*}
\p_i\left(\bp_l\varphi_{\bar{j}}^i-\bp_j\varphi_{\bar{l}}^i\right)
&=\p_i\left(\varphi_{\bar{l}}^k\p_k\varphi_{\bar{j}}^i-\varphi_{\bar{j}}^k\p_k\varphi_{\bar{l}}^i\right)\\\nonumber
&=\varphi_{\bar{l}}^k\p_i\p_k\varphi_{\bar{j}}^i-\varphi_{\bar{j}}^k\p_i\p_k\varphi_{\bar{l}}^i.
\end{align*}

Comparing the two expressions in formulae (\ref{identity-LHS}) and
(\ref{identity-RHS}), we finally obtain the identity
(\ref{identity}).
\end{proof}

\begin{proposition} \label{Theorem-equiv}
Suppose that $(\mathcal{L}_M,h_\omega)$ is a positive line bundle over a
compact K\"ahler manifold $(M,\omega)$ with
$\sqrt{-1}\Theta=\rho\omega$ for a constant $\rho>0$, let
$\varphi\in A^{0,1}(M,T^{1,0}M)$ be an integrable Beltrami
differential which satisfies the conditions that $ div\varphi=0 $ and
$L_\infty$-norm $\|\varphi\|_\infty<1$. Then for any holomorphic
$\sigma_0\in A^{n,0}(M,\mathcal{L}_M)$, a solution of the following
equation
\begin{align} \label{equivalentequ}
\sigma-\sigma_0=-\bar{\partial}^*G\nabla'(\varphi\lrcorner \sigma),
\end{align}
is also a solution of the equation
\begin{align}
\label{extensionequsepecial} \bp \sigma=-\nabla'(\varphi\lrcorner
\sigma)
\end{align}
\end{proposition}

\begin{proof}
Suppose that $\sigma\in A^{n,0}(M,\mathcal{L}_M)$ satisfies the equation
(\ref{equivalentequ}). First, by using the positivity condition for
$\mathcal{L}_M$, we have
\begin{align*}
\bp\sigma&=-\bp\bp^*G\nabla'(\varphi\lrcorner \sigma)\\\nonumber
&=(\bar{\partial}^*\bar{\partial}-\overline{\square})G\nabla'(\varphi\lrcorner
\sigma)\\\nonumber
&=(\bar{\partial}^*\bar{\partial}G-I+H)\nabla'(\varphi\lrcorner
\sigma)\\\nonumber &=-\nabla'(\varphi\lrcorner
\sigma)+\bar{\partial}^*\bar{\partial}G\nabla'(\varphi\lrcorner
\sigma).
\end{align*}
Let $ \Phi=\bar{\partial}\sigma+\nabla'(\varphi\lrcorner \sigma), $
then under the condition $div\varphi=0$, we obtain
\begin{align*}
\bp \nabla'(\varphi\lrcorner \sigma)=-\nabla'(\varphi\lrcorner \Phi)
\end{align*}
by Lemma \ref{keylemma}.

Therefore
\begin{align*}
\Phi&=\bar{\partial}\sigma+\nabla'(\varphi\lrcorner
\sigma)\\\nonumber
&=\bar{\partial}^*\bar{\partial}G\nabla'(\varphi\lrcorner
\sigma)\\\nonumber
&=\bar{\partial}^*G\bar{\partial}\nabla'(\varphi\lrcorner
\sigma)\\\nonumber &=-\bar{\partial}^*G\nabla'\left(\varphi\lrcorner
\Phi\right).
\end{align*}

By quasi-isometry formula (\ref{quasi-isometry-vectorvalue})   and
the condition $||\varphi||_{\infty}<1$, we have
\begin{align*}
||\Phi||^2\leq ||\phi\lrcorner \Phi||^2\leq
||\phi||_\infty||\Phi||^2< ||\Phi||^2,
\end{align*}
 we get the contradiction $\|\Phi\|^2<\|\Phi\|^2$ unless
$\Phi=0$. Hence,
\begin{align*}
\bar{\partial}\sigma=-\nabla'(\varphi\lrcorner \sigma).
\end{align*}
\end{proof}
\begin{remark}
When $(\mathcal{L}_M,h_\omega)$ is semi-positive, we can obtain the
same conclusion as in Proposition \ref{Theorem-equiv}, if we
substitute the global condition $\|\varphi\|_\infty<1$ by requiring
that $\varphi$ (under the H\"older norm as in \cite{MK71}) is
small enough. We leave the further discussion of the extension
equation (\ref{extensionequ})  to another paper.
\end{remark}

By Corollary \ref{Coro-inverse},  it is easy to see that the
equation
\begin{align*}
\sigma-\sigma_0=-\bar{\partial}^*G\nabla'(\varphi\lrcorner
\sigma)=T^{\nabla'}\varphi\lrcorner\sigma
\end{align*}
has a unique solution given by
\begin{align*}
\sigma=(I+T^{\nabla'}\varphi)^{-1}\sigma_0.
\end{align*}

In conclusion,  we obtain
\begin{theorem} \label{theoremglobal}
Suppose that $(\mathcal{L}_M,h_\omega)$ is a positive line bundle over a
compact K\"ahler manifold $(M,\omega)$  with curvature
$\sqrt{-1}\Theta=\rho\omega$ for a constant $\rho>0$, let
$\varphi\in A^{0,1}(M,T^{1,0}M)$ be an integrable Beltrami
differential satisfying the two conditions $div\varphi=0$ and
$L_\infty$-norm $\|\varphi\|_\infty<1$. Then, given any holomorphic
pluricanonical form $\sigma_0\in A^{n,0}(M,\mathcal{L}_M)$, we have that
\begin{align*}
\sigma(\varphi)=\rho_\varphi((I+T^{\nabla'}\varphi)^{-1}\sigma_0)
\end{align*}
is a holomorphic pluricanonical form in
$A^{n,0}(M_\varphi,\mathcal{L}_{M_\varphi})$.
\end{theorem}

Theorem \ref{theoremglobal} gives a closed formula for the extension
of pluricanonical forms. Note that the above construction is global
in the sense that it does not depend on the local deformation theory
of Kodaira-Spencer and Kuranishi \cite{MK71}. The application  of
Theorem \ref{theoremglobal} will be discussed in the following
section.
\subsection{Extension of puricanonical form over
K\"ahler-Einstein manifold of general type}
\label{subsection-KEextensionformula} The invariance of
plurigenera for K\"ahler-Einstein manifold of general type has
already been known, see for example \cite{Sun}. Here we derive an explicit and closed formula as a direct
application of  Theorem \ref{theoremglobal}.

Let $(M,\omega)$ be a K\"ahler manifold. Denote the associated
K\"ahler form by $\omega=\frac{\sqrt{-1}}{2}g_{i\bar{j}}dz^{i}\wedge
d\bar{z}^j$. The corresponding Chern-Ricci form $Ric(\omega)$ is
given by
\begin{align*}
Ric(\omega)=\frac{\sqrt{-1}}{2}\bp\p\log(\det g).
\end{align*}

\begin{definition} \label{def-KEgeneraltype}
We say $(M,\omega)$ is a K\"ahler-Einstein manifold of general type
if $Ric(\omega)=-\omega$.
\end{definition}
In the following discussion, we assume that $(M,\omega)$ is a
K\"ahler-Einstein manifold of general type.
\begin{proposition} [cf. Theorem 1.1 in \cite{Sun} ] \label{gaugeequiv}
Let $\varphi\in A^{0,1}(M,T^{1,0}M)$ be an integrable Betrami
differential, then $\bp^* \varphi=0$ if and only if $div \varphi=0$.
\end{proposition}

Now we consider the deformation of complex structures of
K\"ahler-Einstein manifolds of general type. Let
\begin{align*}
\pi: \mathcal{M}\rightarrow B
\end{align*}
be a holomorphic family of K\"ahler-Einstein manifolds of general
type. Here $B=B_\epsilon\subset \mathbb{C}$ is the open disk of
radius $\epsilon$. Let $t$ be the holomorphic coordinate on $B$. For
$t\in B$, we let $M_t=\pi^{-1}(t)$ be the fiber with the complex
structure induced by the integrable Beltrami differential
$\varphi(t)\in A^{0,1}(M,T^{1,0}M)$, which satisfies
\begin{equation*}
\left\{ \begin{aligned}
         &\bar{\partial} \varphi(t) =\frac{1}{2}[\varphi(t),\varphi(t)]\\
                  &\bar{\partial}^*\varphi(t) =0.
                          \end{aligned} \right.
                          \end{equation*}
where $\p, \bp^*$ are the operators on $M_0$ and $\bp^*$ is defined
with respect to the K\"ahler-Einstein metric $g_0$. We can choose
$\epsilon$ small enough, such that $\|\varphi(t)\|_\infty<1$.

Let $\mathcal{L}_{M_0}=K_{M_0}^{\otimes (m-1)}$ be the holomorphic
line bundle over $M_0$, and the corresponding Hermitian metric be given
by $h_0=(\det g_0)^{-{(m-1)}}$. Let $\nabla=\nabla'+\bp$ be the
Chern connection of $(\mathcal{L}_{M_0},h_0)$. We have that
$\sqrt{-1}\Theta=-2(m-1)Ric(\omega_0)$ by formula
(\ref{curvature-linebundle}).  Recall the operator
$T^{\nabla'}=\bp^*G\nabla'$ we introduced.
\begin{corollary} \label{theoremdeformation}
Given any holomorphic pluricanonical form $\sigma_0\in
A^{n,0}(M_0,\mathcal{L}_{M_0})$, we have that
\begin{align*}
\sigma(t)=\rho_t((I+T^{\nabla'}\varphi(t))^{-1}\sigma_0)
\end{align*}
is a holomorphic pluricanonical form in $A^{n,0}(M_t,\mathcal{L}_{M_t})$
with $\sigma(0)=\sigma_0$, where $\rho_t=\rho_{\varphi(t)}$.
\end{corollary}
\begin{proof}
Since $(M_0,\omega_0)$ is K\"ahler-Einstein of general type, i.e.
$Ric(\omega_0)=-\omega_0$. Hence $\sqrt{-1}\Theta=2(m-1)\omega_0.$
Thus $(\mathcal{L}_{M_0},h_0)$ is a line bundle which satisfies the
conditions in Theorem \ref{theoremglobal}. Then Corollary
\ref{theoremdeformation} followed by Theorem \ref{theoremglobal} and
Proposition \ref{gaugeequiv}.
\end{proof}

By using Corollary \ref{theoremdeformation}, one can  write down the
curvature formula of the induced $L^2$ metric on the generalized
Hodge bundle over the base $B$ with fiber $H^0(M_t, K_{M_t}^{\otimes
m})$ as shown in \cite{Sun}.

\section{$L^2$-Hodge theory and quasi-isometry} \label{Section-L2}
We will first review abstract Hodge theory, then we derive from
\cite{Gromov} that $L^2$-Hodge decomposition theory holds on the
universal cover of a K\"ahler hyperbolic manifold. Then we show the
quasi-isometry formula in $L^2$-Hodge theory, and discuss the
relationship between $L^2$-Hodge theory and the $L^2$-estimate of
H\"ormander.

\subsection{Abstract Hodge theory} For the basics of the $L^2$-Hodge theory,
we refer the reader to Chapter 2 in \cite{Ohsawa} or Appendix C in
\cite{Ballmann}. Let $H_1, H_2$ and $H_3$ be three Hilbert spaces.
Let $T: H_1 \rightarrow H_2$ and $S: H_2 \rightarrow H_3$ be densely
defined closed operators. We assume
\begin{eqnarray*}
\text{Im}(T) \subset \text{Dom}(S)
\end{eqnarray*}
and
\begin{eqnarray*}
STx =0 ~~~~ \text{for} ~~~~ x \in \text{Dom}(T).
\end{eqnarray*}
\noindent Then their Hilbert space adjoint operators $S^*: H_3
\rightarrow H_2$ and $T^*: H_2 \rightarrow H_1$ satisfy the same
conditions, namely,
\begin{eqnarray*}
\text{Im}(S^*) \subset \text{Dom}(T^*),
\end{eqnarray*}
and
\begin{eqnarray*}
T^* S^* x=0 ~~~~ \text{for}~~~~ x \in \text{Dom}(S^*).
\end{eqnarray*}

\noindent If we define $\Delta: H_2 \rightarrow H_2$ by
\begin{eqnarray*}
&& \text{Dom}(\Delta) = \{ x \in \text{Dom}(S) \cap \text{Dom}(T^*); T^* x \in \text{Dom}(T), Sx \in \text{Dom}(S^*)\},\\
&&\Delta s = S^* S x + T T^* x ~~~~\text{for}~~~~x \in
\text{Dom}(\Delta).
\end{eqnarray*}

\noindent Then we have the following existence theorem for the Green
operator of the operator $\Delta$, as given in \cite{KK87}. We
summarize it here for reader's convenience.

\begin{proposition}\label{abstractHodge}
Let $S, T$ and $\Delta$ be as above, and suppose that both
$\text{Im}(S)$ and $\text{Im}(T)$ are closed. Denoting
$\Ker(\Delta)$ by $\mathcal{H}$, we have the folllowing:
\begin{enumerate}
\item $\Delta$ is self-adjoint, that is $\Delta= \Delta^*$ holds.
\item $\mathcal{H} = \Ker(\Delta) = \Ker(T^*) \cap \Ker(S)$ and $\mathcal{H}^{\bot}= \text{Im}(\Delta)$.
\item Denoting by $P_{\mathcal{H}}$ the projection operator: $H_2 \rightarrow \mathcal{H}$,
the Green operator $$G=(\Delta|_{\mathcal{H}^{\bot}})^{-1}(I -
P_{\mathcal{H}})$$ is well-defined and bounded.
\end{enumerate}
\end{proposition}

\subsection{K\"ahler hyperbolicity and $L^2$-Hodge theory}
Let $M$ be a compact K\"ahler manifold with a K\"ahler form
$\omega$. Let $\pi:\, \tilde{M}\to M$ be the universal cover of $M$.
Then $M$ is called K\"ahler hyperbolic as defined in \cite{Gromov},
if
$$\pi^*\omega =d\alpha$$ with $L_\infty$-norm $\|\alpha\|_\infty$ finite under the metric $\pi^*\omega$ on $\tilde{M}$.

We are interested in the $L^2$-cohomology on $\tilde{M}$ with the
metric $\tilde{\omega}=\pi^*\omega$.

\begin{lemma}\label{L21} There exists $L^2$-Hodge theory on $\tilde{M}$ with the complete K\"ahler metric $\tilde{\omega}= \pi^*\omega$.
More precisely complete Hodge decompositions holds on $\tilde{M}$ with
K\"ahler form $\tilde{\omega}$.
\end{lemma}
\begin{proof} This follows from the estimates in Theorem 1.4.A in \cite{Gromov}. Indeed, let $\triangle$ denote one of the
three Laplacians
$$\Box_{\bp} = \bp\bp^*+\bp^*\bp, \ \Box_{\p} = \p\p^*+\p^*\p,\, \mathrm{and}\  \triangle_d = dd^* +d^*d.$$
 Then the estimates in Theorem 1.4.A in \cite{Gromov} gives us that,
for any $L^2$-form $\psi$ that is orthogonal to the harmonic forms,
we have
$$\langle\psi, \triangle\psi\rangle\geq\lambda_0^2\langle\psi,\psi\rangle$$
where $\lambda_0$ is a strictly positive constant which only depends
on the dimension of $M$ and the bound on $\alpha$.

Since $\tilde{\omega}$ is a complete metric, as discussed in
Appendix C in \cite{Ballmann}, we know that
$$\mathrm{ker}\, \Box_{\bp}= \mathrm{ker}\,\bp\cap \mathrm{ker}\,\bp^*,$$
similarly for the operators $\p$ and $d$.
 By Theorem 2.1 and Theorem 2.2 in Chapter 2.1.2 of  \cite{Ohsawa}, we know that the images of the operators $\bp$, $\p$ and $d$ are closed.
  Therefore from Proposition \ref{abstractHodge}, we know that the abstract Hodge theory holds for these operators.

 On the other hand, since $\tilde{\omega}$ is K\"ahler, we have the equalities for the Laplacians.
  $$\Box_{\bp}= \Box_{\p} =\frac{1}{2} \triangle_d$$
 which implies that complete Hodge decomposition theory holds for the $L^2$ cohomology on $\tilde{X}$ with the metric
 $\tilde{\omega}=\pi^*\omega$.
 \end{proof}
 Let $H$ denote the orthogonal projection from the $L^2$-completion of smooth $(p,q)$-forms on $\tilde{M}$, $L^{p,q}_2(\tilde{M}, \tilde{\omega})$,
 to the harmonic space $\mathbb{H}=\ker\,\Box_{\bp}$. As a corollary of the abstract Hodge decomposition in Proposition \ref{abstractHodge}, we have

 \begin{corollary}\label{hodge} On the universal cover $\tilde{M}$ of the K\"ahler hyperbolic manifold $M$ with
 complete K\"ahler metric $\tilde{\omega}$, there exists a bounded operator $G$ on $L^{p,q}_2(\tilde{M}, \tilde{\omega})$,
 called the Green operator such that
 \begin{equation}
\square_{\overline{\partial}} G = G \square_{\overline{\partial}}=
Id -H,
 HG=G H=0,
\end{equation}
\end{corollary}

Therefore we have the following quasi-isometry formula in
$L^2$-Hodge theory. Its proof is completely the same as in the case
for compact K\"ahler manifold as given in Section
\ref{Section-Hodgetheory}.
\begin{corollary}\label{quasiisometry}
For $g\in \text{Dom}(\p) \subset L^{p,q}_2(\tilde{M},
\tilde{\omega})$, we have that
\begin{align}
\|\bar{\partial}^* G\partial g\|^2\leq \|g\|^2.\notag
\end{align}
\end{corollary}

We consider the operator $$T=\bp^* G\p$$ in $L^2$ Hodge theory. Then
Corollary \ref{quasiisometry} tells us that $T$ is an operator of
norm less than or equal to $1$,  in the Hilbert space of
$L^2$-forms. So we have the following

\bcorollary\label{norm<1} On the complete K\"ahler manifold
$(\tilde{M},\tilde{\omega})$ discussed above, let $\phi$ be a
Beltrami differential acting on the Hilbert space of $L^2$-forms by
contraction such that its $L_\infty$-norm $||\phi||_\infty<1$, then
the operator $I+T\phi$ is invertible. \ecorollary

\subsection{Relation to $L^2$-estimates}
We give a brief discussion of the results in this section to the
$L^2$-estimates of H\"omander. We believe that this point of view is
 more geometric and will be useful in studying the optimal constants
 in $L^2$-estimate problems. This subsection is independent of the
 rest of this article.

 We consider a complete K\"ahler manifold
 $M$ with a holomorphic vector bundle $E$. First note that the same argument of this section, applied to a
bundle version of $L^2$-Hodge theory, shows that
$$\|\bar{\partial}^*G g\|^2= \langle \bar{\partial}^*Gg,\bar{\partial}^*G g\rangle
=\langle\bar{\partial}\bar{\partial}^*G g, G g\rangle\leq \langle
Gg, g\rangle.$$ See \cite{LRY} for the case of compact K\"ahler
manifold. Here we have used the identity
$$\langle\bar{\partial}\bar{\partial}^*G g, G g\rangle
= \langle \Box_{\bp}G g, g\rangle-\langle \bp^*\bp G g, g\rangle =
\langle G g, g\rangle-\langle H g,Hg\rangle - \langle G\bp g, G\bp
g\rangle$$ to derive the above inequality. From this we see that if
$\bp g=0$ and $H g=0$, then $$f= \bp^* Gg$$ solves the equation $\bp
f =g$ with $L^2$-estimate.  From these we can see the interesting relation to the
$L^2$-estimate of H\"ormander in solving the $\bp$-equation $\bp f =
g$.

Furthermore the curvature condition on the vector bundle  $E$ in
H\"ormander's $L^2$-estimate and the Bochner-Kodaira-Nakano identity
imply that
$$\langle \Box_{\bp}\, g, g\rangle \geq \langle Ag, g\rangle$$ with a constant $A>0$  which implies that, if the metric is complete,  $L^2$-Hodge theory holds by the same argument as the proof of Lemma \ref{L21}, and in the meantime gives us the needed $L^2$-estimate
$$\langle G g, g\rangle \leq \langle A^{-1} g, g\rangle.$$
We refer the reader to \cite{Dem} for details about H\"omander's
$L^2$-estimates.



\section{Solving the Beltrami equations}\label{Beltrami}
In this section we will apply $L^2$-Hodge theory on the unit disc to
solve the classical Beltrami equations for quasi-conformal maps as
discussed in \cite{Ahlfors}, \cite{Bojarski} and \cite{Gut}. The
problem of solving the Beltrami equations has a long history in
geometry and analysis, and it has important applications in complex
and analytic geometry.  Many methods have been developed for solving
such equations. For some of its many important applications, see
\cite{Bers}. We will show that Hodge theory method gives a much
simpler way to solve these equations. Our method can be viewed as a
more geometric way to treat these equations.

 Recall that the Beltrami equation is to solve for a function
$f:\ D \to \mathbb{C}$ where $D$ is a unit disc in the complex
plane, such that
$$f_{\bar{z}} = \mu_0 f_{{z}}.$$ Here $z$ is the variable in $D$
and $\mu_0$ is some function on $D$ with $\mathrm{sup}\,|\mu_0|<1$.
Let $\mu = \mu_0 \frac{\p}{\p z} \otimes d\bar{z}$, we can rewrite
the Beltrami equation in our familiar form,
$$\bp f  = \mu\p f.$$
We will assume that $\mu$ is measurable, or equivalently $\mu_0$ is a
measurable function.

Let us consider the unit disc $D$ with the standard Poincar\'e
metric $\omega_P$ of curvature $-1$, which clearly satisfies the
condition in Lemma \ref{L21}. Therefore,  $L^2$-Hodge theory holds
on $(D,\omega_P)$.

 First note that the
$L_\infty$-norm of $\mu$ with the Poincar\'e metric is actually the
same as its $L_\infty$-norm with the Euclidean metric on $D$,
therefore is the same as $\mathrm{sup}\,|\mu_0|$ where the supremum
is taken on the unit disc $D$. Also note that the $L^2$-norm of a
holomorphic one form on $D$ is independent of the Hermitian metric
on $D$. These are crucial for our applications of $L^2$-Hodge theory
on $D$. We have the following result for measurable $\mu$ or
$\mu_0$.
\begin{proposition} Assume that the $L_\infty$-norm of $\mu$, $||\mu||_{\infty}<1$.
Then given any holomorphic one form $h_0$ on $D$, the equation
 \begin{equation}\label{Bel}
 \bp h =- \p \mu h
 \end{equation}
 has a solution
\begin{align} \label{Belsolution}
h =(I + T\mu)^{-1} h_0.
\end{align}
\end{proposition}
\bproof

The proof of Proposition \ref{Bel} can be given by using Corollary
\ref{norm<1} in the same way as in the proof of Proposition
\ref{prop1-section} as shown in Section
\ref{Section-extensionforms}. Here we will give another direct
method to show that $(\ref{Belsolution})$ is the solution to
equation (\ref{Bel}).

Plug $h= h_0+h_1$ into the equation (\ref{Bel}), then
$$\bp h_1 = -\p \mu  h_0-\p \mu h_1.$$ Here $\mu$ as above is a simplified notation for the contraction operator $\mu\lrcorner$. Applying operator $\bp^*G$ and Hodge theory to both
sides, we get
$$h_1 -Hh_1 = -T\mu h_0-T\mu h_1.$$

We consider the equation by assuming $Hh_1=0$,
$$(I+T\mu)h_1 = -T\mu h_0.$$
By Corollary \ref{norm<1}, we know that when the $L_\infty$-norm of
$\mu$, $||\mu||_\infty< 1$, the operator $(I+T\mu)$ is invertible.
Let
$$h_1 = -(I+T\mu)^{-1} T\mu h_0.$$Then it is an easy exercise to check directly that
$h= h_0 +h_1$ is the same as formula (\ref{Belsolution}) which
gives a solution of Equation (\ref{Bel}) which is a closed formula.

Indeed, since $h=h_0+h_1$ is a $(1,0)$-form, and $ (I+T\mu)h_1
=-T\mu h_0$ which is $$h_1 = -T\mu (h_0+h_1)=-T\mu h.$$ We have
$$\bp T\mu h= \bp \bp^*G\p \mu h= \Box_{\bp} G \p \mu h=(I-H) \p \mu h=\p \mu h.$$
Then
$$\bp h_1 = -\p \mu (h_0+h_1)= -\p\mu h_0 -\p\mu h_1$$ which can be rewritten as
$$\bp h = -\p \mu h.$$
 \eproof

\btheorem  Assume that $||\mu||_\infty=\mathrm{sup}\, |\mu_0|<1$, if
$\mu_0$ is of regularity $C^k$,  then the Beltrami equation
$$\bp f  = \mu\p f$$ has a solution $w$ of regularity $C^{k+1}$.
\etheorem
\begin{proof}
By Proposition \ref{Bel}, we can find a $(1,0)$-form $h$ of
regularity $C^k$ on $D$, which satisfies the equation
\begin{align*}
\bar{\partial}h=-\partial \mu h,
\end{align*}
by noting that the Green operator maps a $C^k$ form to a $C^{k+2}$
form.

It follows that
\begin{align*}
d(e^{i_{\mu}}h)=0.
\end{align*}

According to Poincar\'e lemma, there is a function $w$ of regularity
$C^{k+1}$ on $D$, such that
\begin{align*}
e^{i_{\mu}}h=dw=\bar{\partial} w+\partial w.
\end{align*}
Since
\begin{align*}
e^{i_{\mu}}h=h+\mu h,
\end{align*}
by considering the types, we obtain
\begin{align*}
h=\partial w \ \text{and}\  \mu  h=\bp w.
\end{align*}
Therefore
\begin{align*}
\bp w=\mu  \p w.
\end{align*}
\end{proof}

\section{Generalizations and open questions}\label{general}

Although we only consider holomorphic $(n,0)$-forms in this paper,
our method also works for a general $(p,q)$-form $\sigma_0$ with
$d\sigma_0=0$. In this case, the equation for an extension is
$$d\sigma  -\sL^{1,0}_{\phi}\sigma=0$$ as follows from the formula in Corollary \ref{commuted1}.
Our method reduces the equation to the system of equations,

\begin{eqnarray}\label{induction}
\begin{cases}
\p \sigma=0,\\
\bp \sigma  =-\p(\phi\lrcorner \sigma),
\end{cases}
\end{eqnarray}
and $\rho(\sigma)=e^{i_\phi} \sigma$ gives an extension of
$\sigma_0$ as $d$ closed forms on $M_\varphi$. The closed formula we
derive for the extensions of holomorphic forms applies equally well
to $(p,q)$-forms. From this one can give a direct proof of the
invariance of Hodge numbers for deformations of compact K\"ahler
manifolds, by showing that any $d$-closed $(p,q)$-form always has an
extension by using Hodge theory. From this one may also derive a
simpler proof of the stability of the deformation of K\"ahler
manifolds. One may compare this approach with \cite{RWZ}. A related challenging problem is
to prove the invariance of plurigenera for compact K\"ahler
manifolds \cite{Siu} by using the method in Section
\ref{Section-pluricanonical}, we leave the further discussion of
this topic for another paper.

For the case of complete K\"ahler manifold with Poincar\'e type
metric on a quasi-projective manifold, the results of \cite{CKS87}
and \cite{KK87} tells us that $L^2$ Hodge decomposition theory
holds, therefore we also have similar closed explicit formulas for extensions
of $(p,q)$-forms similar to the formulas for compact K\"ahler
manifold. In particular for Riemann surface with punctures and
Poincar\'e metric of constant negative curvature $-1$, our method
gives closed formulas for extensions and global sections of
holomorphic one forms on the Teichm\"uller or Torelli space of
punctured Riemann surfaces.

Since the Green operator is an integral operator, we may also apply
the Hodge theory or $L^2$-Hodge theory method as discussed in this
paper to study various integral formulas and extension formulas in
complex analysis, which will supply a unifying geometric treatment
of some classical formulas. This will be another interesting topic
we hope to study in the future.

\end{document}